\documentclass[12pt]{amsart}

\usepackage[all]{xy}
\usepackage{xcolor}

\definecolor{StatusSub}{HTML}{0072B2}   
\definecolor{StatusNot}{HTML}{F0E442}   
\definecolor{StatusOpen}{HTML}{666666}  
\definecolor{amethyst}{rgb}{0.6, 0.4, 0.8}
\usepackage{color}
 \usepackage{adjustbox}
\usepackage[pagebackref=true, colorlinks, allcolors=amethyst]{hyperref}

\renewcommand*{\backref}[1]{}
\renewcommand*{\backrefalt}[4]{%
  \ifcase #1 %
    \relax
  \or
    $\uparrow$#2.%
  \else
    $\uparrow$#2.%
  \fi%
}

\usepackage{mathtools}

\usepackage{fullpage,url,amssymb,colonequals,algorithm2e}
\usepackage[shortlabels]{enumitem}

\usepackage{mathrsfs} 
\usepackage{microtype}
\usepackage{tabularx} 
\usepackage{booktabs} 
\usepackage{tikz-cd,tikz}
\usepackage{pgfplots}
\usepackage{pgfplotstable}
\usepackage{caption}
\pgfplotsset{compat=1.18}
\usepgfplotslibrary{groupplots}

\newcommand{\testbar}[2]{%
  \pgfmathsetmacro{\barwidth}{max(0.6, 100*ln(#2)/ln(9094))}%
  \begin{tikzpicture}[baseline=-0.55ex]
    \fill[#1!10] (0,0) rectangle (2.7,0.16);
    \draw[black!45, line width=0.2pt] (0,0) rectangle (2.7,0.16);
    \fill[#1!75] (0,0) rectangle ({2.7*\barwidth/100},0.16);
  \end{tikzpicture}%
}

\newcommand{\F}{\mathbb{F}}

\newcommand{\N}{\mathbb{N}}

\renewcommand{\P}{\mathbb{P}}
\newcommand{\Q}{\mathbb{Q}}
\newcommand{\R}{\mathbb{R}}
\newcommand{\Z}{\mathbb{Z}}

\newcommand{\Qbar}{{\overline{\Q}}}

\newcommand{\Fbar}{{\overline{\F}}}

\newcommand{\calO}{\mathcal{O}}
\newcommand{\calP}{\mathcal{P}}

\newcommand{\calR}{\mathcal{R}}
\newcommand{\calS}{\mathcal{S}}

\DeclareMathOperator{\alg}{alg}

\DeclareMathOperator{\Aut}{Aut}

\DeclareMathOperator{\CH}{CH}

\DeclareMathOperator{\diag}{diag}

\DeclareMathOperator{\End}{End}

\DeclareMathOperator{\Frob}{Frob}
\DeclareMathOperator{\Gal}{Gal}

\DeclareMathOperator{\JL}{JL}

\DeclareMathOperator{\nrd}{nrd}

\DeclareMathOperator{\Sh}{Sh}
\DeclareMathOperator{\new}{new}

\DeclareMathOperator{\Orth}{O}

\DeclareMathOperator{\pr}{pr}

\DeclareMathOperator{\tr}{tr}

\DeclareMathOperator{\Tr}{Tr}

\newcommand{\PGL}{\operatorname{PGL}}
\newcommand{\PSL}{\operatorname{PSL}}
\newcommand{\SL}{\operatorname{SL}}

\newcommand{\Dnew}{D\operatorname{-new}}

\newcommand{\notdiv}{\nmid}

\newcommand*{\SheafIsom}{\mathrm{I}\kern -.5pt som}

\numberwithin{equation}{subsection}
\newtheorem{theorem}[equation]{Theorem} 
\theoremstyle{plain}
\theoremstyle{plain}

\newtheorem{lemma}[equation]{Lemma}
\newtheorem{corollary}[equation]{Corollary}
\newtheorem{proposition}[equation]{Proposition}

\theoremstyle{definition}
\newtheorem{definition}[equation]{Definition}
\newtheorem{question}[]{Question}

\newtheorem{example}[equation]{Example}

\theoremstyle{remark}
\newtheorem{remark}[equation]{Remark}

\newcommand{\Order}{\calO} 
\newcommand{\Curve}[2]{X_0({#1},{#2})} 
\newcommand{\StarCurve}[2]{X_0^*({#1},{#2})} 
\newcommand{\ALfull}[2]{W_0({#1},{#2})} 
\newcommand{\gon}[2]{\gamma_{#2}(#1)}
\newcommand{\reldeg}[2]{d_{{#1}, {#2}}}

\newcommand{\defi}[1]{\emph{#1}} 

\newcommand{\mattwo}[4]{\left( \begin{array}{cc} {#1} & {#2} \\ {#3} & {#4} \end{array} \right)}
\newcommand{\vectwo}[2]{\left( \begin{array}{c} {#1} \\ {#2} \end{array} \right)}

\allowdisplaybreaks

\newcommand{\Magma}{\textsc{Magma}}

\begin{document}
\title{Hyperelliptic Atkin--Lehner quotients of Shimura curves}

\author{Eran Assaf and Sachi Hashimoto}

\begin{abstract}
We work towards completely classifying all hyperelliptic Atkin--Lehner quotients of Shimura curves $\Curve{D}{N}/W$ with level $N$ coprime to $D$ and $W \le \ALfull{D}{N}$, extending, on the one hand, a result of Ogg that provided such a classification for the trivial quotients (the case $W = 1$), and on the other hand, results of Furumoto and Hasegawa that provided such a classification for modular curves (the case $D = 1$). As a byproduct of our methods, building on the works of Guo and Yang, we also obtain models for some quotients of genus at most two, answering some questions of Padurariu and Saia.
\end{abstract}

\maketitle

\section{Introduction}

\subsection{Background and motivation}
Let $D, N$ be positive integers, such that $\mu(D) = 1$ (i.e. $D$ is squarefree with an even number of prime factors) and $(D,N) = 1$. Let $B = B_D$ be the quaternion algebra of discriminant $D$ over $\Q$, and let $\Order \subseteq B$ be an Eichler order of level $N$. The Shimura curves $\Curve{D}{N}$ classify abelian surfaces with (potential) quaternionic multiplication by $\Order$ and the study of their low degree points has a long history, beginning with \cite{Ogg}, who showed they only have rational points for $D = 1$. 

A basic invariant of a curve $X$ is its gonality $\gon{X}{}$ -- the least degree of a map from the curve to $\P^1$. For example, the gonality $2$ curves are the (geometrically) hyperelliptic curves. We say a curve is \defi{subhyperelliptic} if its gonality is at most $2$.
In \cite{OggHyperelliptic}, Ogg classified which of the curves $\Curve{D}{N}$ are hyperelliptic. Following this line of research, Padurariu and Saia, in \cite{OanaFreddyBielliptic}, classified the trigonal ($\gon{X}{} = 3$) ones and began to classify bielliptic such curves, using this classification to determine which ones have infinitely many quadratic points.

The curve $\Curve{D}{N}$ admits a natural action by the associated Atkin--Lehner group $\ALfull{D}{N} = N_{B^{\times}}(\Order) / \Order^{\times}$, leading to another natural source for low degree points, namely the quotient curves $\Curve{D}{N}/W$ for subgroups $W \le \ALfull{D}{N}$. 
Study of the rational points on these curves began in \cite{OanaCiaran}, where the authors find rational points when $\Curve{D}{N}$ is hyperelliptic. However, the broader question of classifying the hyperelliptic quotients $\Curve{D}{N}/W$  remains open when $D > 1$. In this work, we take the first steps towards answering this question,
\begin{question}
    For which values of $D$ and $N$, and which subgroups $W \le \ALfull{D}{N}$ is the curve $\Curve{D}{N}/W$ geometrically hyperelliptic?
\end{question}
The case $D = 1$, where the curve $\Curve{1}{N} \simeq X_0(N)$ is isomorphic to the classical modular curve, classifying cyclic $N$-isogenies between elliptic curves, has been treated in a series of papers by Hasegawa and Hashimoto \cite{HH96}, Hasegawa \cite{HasegawaStarcurves}, and Furumoto and Hasegawa \cite{FH99}. They reach a full classification by applying a variety of methods, some of which directly use the $q$-expansions of classical modular forms and explicit equations embedding the curves in projective space. 

Interestingly, the case $D > 1$ has proved greater resistance to such methods. Although methods for producing $q$-expansions of quaternionic modular forms around CM points have been described in the literature \cite{VW14, Nelson}, in practice, a robust rigorous version has not been implemented, and so one cannot obtain results that are provably correct.

There have been numerous attempts to produce explicit equations for Shimura curves with $D > 1$, see \cite{Takeuchi, Elkies, GRgenus2, GRgenus1, Yang, OanaFreddy},
constructing models for curves of genus $\le 2$, and \cite{Molina, GY17}, producing equations for hyperelliptic curves. 
However, since the methods of \cite{Molina} and \cite{GY17} require knowledge of the hyperelliptic map, one cannot use them directly to detect hyperellipticity. Moreover, none of these methods have been robustly implemented.

This paper generalizes the methods of \cite{HH96,HasegawaStarcurves,FH99} for $D > 1$, attempting to classify as many curves as possible without resorting to explicit equations. As some of the methods require explicit equations for some genus $0$ quotients, we have also implemented a version of the method described in \cite{GY17}, using Borcherds products and Schofer's formula to construct equations for these quotients. These models, in turn, answer some open questions raised in \cite{OanaFreddy}.

\subsection{Results}

We begin by reducing the question to a finite problem.

\begin{theorem}
\label{thm:finitenessmain}
    There are finitely many $D,N,W$ such that $\Curve{D}{N}/W$ is hyperelliptic. 
\end{theorem}
We prove this using point counts over finite fields. 
There is an upper bound on the number of points of a hyperelliptic curve coming from the fact that it possesses a degree 2 map $X \to \P^1$. On the other hand, we have a lower bound from supersingular points. These bounds cross for large enough $(D,N)$, reducing the computation to $2342$ pairs of $(D,N)$. 

The main result of our paper is the following more explicit description of the set of (geometrically) hyperelliptic curves.

\begin{theorem}
    There exist sets $\calS_h$ and $\calS_u$ with $\# \calS_h = 3423$ and $\# \calS_u = 4207$, and a set $\calS$ with $\calS_h \subseteq \calS \subseteq \calS_u$ such that
    $X = \Curve{D}{N}/W$ is hyperelliptic if and only if $(D,N,W) \in \calS$. 
\end{theorem}

The sets $\calS_h$ and $\calS_u$ describe a set of curves that are proved to be hyperelliptic, and a candidate set of curves containing all hyperelliptic curves, respectively, are given explicitly, and can be found on our GitHub repository \cite{GitHub} at \url{https://github.com/assaferan/ShimuraCurveALQuotients/blob/v1.2/data/curves_after_UpdateCurves8.dat}. 

We also implement a version of the method \cite{GY17} to construct equations for Shimura curves that admit a double cover of $\P^1$. 
As a result, we answer some open questions and construct models for Shimura curves appearing in \cite{OanaFreddy}, which were hitherto unknown (see Tables \ref{modelsgen2} and \ref{modelsbielliptic}). The following proposition illustrates this application.

\begin{proposition}[Proposition~\ref{prop:rational genus 0}]
For $(m,d) \in \{ (2,5), (5,7), (10,14) \}$, we have
\begin{equation}
    \Curve{10}{7}/\langle w_{m}, w_{d}\rangle \simeq \P^1_{\Q}.
\end{equation}
In particular, these genus $0$ curves all admit rational points.
\end{proposition}

For reproducibility of our results, the code that has been used to implement our methods in \Magma{} \cite{Magma} is freely available online at \url{https://github.com/assaferan/ShimuraCurveALQuotients}.

\subsection{Overall strategy}

The broad strategy is to first reduce to a finite list of curves, then filter the list of candidates using a variety of techniques.

The covering structure of the Shimura curves $\Curve{D}{N}/W$ plays an important role in determining and propagating the subhyperellipticity.
An important basic application of the covering is closure principles, which are used throughout. The following two principles are logically equivalent, but will be applied in different settings.
\begin{proposition}[Closure principles]
\label{def:closure}
\hfill
\begin{enumerate}
\item (``Downward closure'')  If $X$ is subhyperelliptic then any covered curve $X \twoheadrightarrow Y$ is also subhyperelliptic.
\item (``Upward closure'') If $Y$ is not subhyperelliptic then any cover $X \twoheadrightarrow Y$ is not subhyperelliptic.
\end{enumerate}
\end{proposition}

\begin{proof}
    This follows from $\gon{Y}{} \le \gon{X}{}$, see \cite[Prop. A.1.(vii)]{Poonen}.
\end{proof}

The \defi{star curves} $\StarCurve{D}{N} := \Curve{D}{N}/\ALfull{D}{N}$ are minimal with respect to the covering structure.
If a star curve is not subhyperelliptic, by upward closure, no cover  is.
We apply this in the following way. To reduce to a finite problem, we first reduce the list of star curves to the list of potentially subhyperelliptic star curves. 
By upward closure, we can then consider only the covers of the subhyperelliptic star curves. The finiteness of the subhyperelliptic star curve candidates also shows  finiteness of all candidates $\Curve{D}{N}/W$.

After reducing to a finite list of star curves we apply a variety of methods to determine if a star curve is hyperelliptic or not. 
Then we generate all covers of the  list subhyperelliptic and potentially hyperelliptic star curves.  
We use various methods to determine if curves in this bigger list are hyperelliptic or not.
At the end of the process, we are left with a list of subhyperelliptic curves, and a list of remaining candidates which are potentially hyperelliptic but whose status remains unknown.

Since $X_0(N) \simeq \Curve{1}{N}$ we take advantage of the opportunity to verify the results of our code against the results and tables in \cite{HH96, HasegawaStarcurves, FH99}. 
In particular, in this paper we include the case $D = 1$ everywhere.

\subsection{Overview of paper}

In Section \ref{sec:reduction} we prove the reduction to a finite list of pairs $(D,N)$, which proves Theorem \ref{thm:finitenessmain}. Next, Section \ref{sec:genusfixedpoints} recalls  background material about optimal embeddings and fixed points of Atkin--Lehner involutions in order to derive a genus formula for the curves $\Curve{D}{N}/W$.  In this section we also introduce results deriving from a proposition of Ogg which filter out non-hyperelliptic curves by counting fixed points of involutions.
In Section \ref{sec:trace} we use trace formulas to count points over finite fields. We also introduce the extra modular non-Atkin--Lehner involutions that appear when $4$ or $9$ divide the level $N$. However, we defer the trace formulas for these other operators to Appendix \ref{appendix}. Section \ref{sec:covers} discusses several tools for determining hyperellipticity that leverage the covering structure of the curves  $\Curve{D}{N}/W$. In Section \ref{sec:isomorphism} we discuss several techniques for forming isomorphisms between these curves and in Section \ref{sec:Weilpolys} we discuss the application of Weil polynomials.

We discuss the results of applying all of these tests to the candidate curves in Section \ref{sec:results}. We also discuss potential methods for future study and limitations of the current method. Finally in Section \ref{sec:models} we discuss an implementation of the methods in \cite{GY17} to compute models of low genus Shimura curves, and use this to answer some questions of \cite{OanaFreddy}.

\subsection{Acknowledgments}
 We delight in thanking Edgar Costa for referring us to \cite{CostaWeilPolys} and assisting us with implementation aspects.
 It is our pleasure to thank Everett Howe for helpful correspondence regarding Weil polynomials of hyperelliptic curves, Bjorn Poonen for communicating to us several proofs of Proposition~\ref{prop: genus 3 covering genus 2} and other insightful comments, and Yifan Yang for helpful correspondence regarding Borcherds forms.

E.A. was supported by a grant from the Simons Foundation (SFI-MPS-Infrastructure-00008651, AS).
S.H. was supported by the National Science Foundation under Award No. DMS-2501658.

\section{Reduction to a finite problem}
\label{sec:reduction}
We first describe how to reduce the problem to a finite list of star curves $\StarCurve{D}{N}$.
The overarching strategy is to obtain lower and upper bounds on the number of points over a finite field, assuming $\StarCurve{D}{N}$ is hyperelliptic. 
These bounds will eventually cross when $D$ and $N$ are large enough.
Given a finite list of candidate hyperelliptic $\StarCurve{D}{N}$, this implies that there are finitely many candidate hyperelliptic $\Curve{D}{N}/W$, since every such curve covers a star curve and by the principle of upward closure (as in Proposition \ref{def:closure}).

\subsection{Notation} Before stating the bounds, we introduce some notation.

For a positive integer $N$, let $\omega(N)$ denote the number of its prime divisors. Write 
\begin{equation} \label{eqn:P1}
\psi(N) \coloneqq \# \P^1(\Z / N \Z) = N \prod_{p \mid N} \left ( 1 + \frac{1}{p} \right),
\end{equation}
where $p$ runs over the prime divisors of $N$,
and let $\varphi(N)$ denote Euler's totient function.

For a curve $X$ defined over a field $K$, we define $\gon{X}{K}$ to be the $K$-gonality of $X$, i.e. the minimum degree of a map $X \to \P^1$ defined over $K$. 

Let $h(N)$ be the largest divisor of $24$ with $h^2 \mid N$. Write $h(N) = h_2(N) h_3(N)$ with $h_2(N) \mid 8$ and $h_3(N) \mid 3$. If $2 \mid h_2(N)^2 \parallel N$, let $s_2(N) = \tfrac{3}{4}$, otherwise set $s_2(N) = 1$. If $h_3(N)^2 = 9 \parallel N$, let $s_3(N) = \tfrac{2}{3}$, otherwise set $s_3(N) = 1$, and let $s(N) = s_2(N) s_3(N)$.

Let $\delta_D = \delta_{D,1}$ be $0$ unless $D = 1$, when it is $1$; 
let $\delta_{4 \mid DN}$ be $0$ unless $4 \mid DN$, when it is $1$, and similarly for $\delta_{9 \mid DN}$. 

Finally, for positive integers $D,N$, let $h(D,N)$ be the class number of an Eichler order of level $N$ in a quaternion algebra of discriminant $D$. Then, by Eichler's class number formula (see e.g. \cite[Theorem 30.1.5]{JV}), one has
    \begin{equation} \label{eq: Eichler class number}
    h(D, N) = \frac{1}{12} \varphi(D) \psi(N) + \frac{\epsilon_2}{4} + \frac{\epsilon_3}{3},
    \end{equation} 
    where
    $$
    \epsilon_2 = \delta_{4 \mid DN} \prod_{p \mid D} \left(1 - \left(\frac{-4}{p} \right) \right)
    \prod_{p \mid N}  \left(1 + \left(\frac{-4}{p} \right) \right), 
    $$
    and
    $$
    \epsilon_3 = \delta_{9 \mid DN} \prod_{p \mid D} \left(1 - \left(\frac{-3}{p} \right) \right)
    \prod_{p \mid N}  \left(1 + \left(\frac{-3}{p} \right) \right). 
    $$

\subsection{Lower and upper bounds}

The following proposition will be our main tool to reduce to a finite problem, having lower bounds from supersingular points and upper bounds from the double cover of $\P^1$. 

\begin{proposition} \label{prop: finite bounds}
    Let $D, N$ be positive integers. Let $W \subseteq \ALfull{D}{N}$ be a subgroup of index $[\ALfull{D}{N} : W] = 2^r$, and let $X = \Curve{D}{N} / W$. 
    Let $p$ be a prime number and set
    $$
    H_p(D,N) = \begin{cases}
        h(Dp, N) & p \nmid DN, \\
        h(Dp, N/p) & p \mid N, \\
        h(D/p, Np) & p \mid D.
    \end{cases}
    $$
    If one has 
    \begin{equation} \label{eq: finite bounds}
    \frac{1}{2^{\omega(DN)-r}} H_p(D,N) + \delta_{D} 2^rh(N)s(N) > 2(p^2+1),
    \end{equation}
    then $\gamma_{\Qbar}(X) > 2$. 
\end{proposition}

\begin{proof}
    Assume, on the contrary, that $\gamma_{\Qbar}(X) \le 2$. Then reducing modulo a prime above $p$ we get $\gamma_{\Fbar_p}(X) \le 2$. 
    Since $X(\F_{p^2}) \ne \emptyset$, by \cite[Theorem 2.5]{Poonen}, we have $\gamma_{\F_{p^2}}(X) = \gamma_{\Fbar_p}(X) \le 2$.
    In particular, we have a map $X \to \P^1$ defined over $\F_{p^2}$, of degree $d \le 2$, so that 
    \begin{equation}
    \label{bound:finitefield}
    \#X(\F_{p^2}) \le d \# \P^1(\F_{p^2}) \le 2(p^2 + 1).
    \end{equation}
    By \cite[Theorem 3.4]{Ribet89SS} (when $p \nmid DN$), \cite[Theorem 4.7(v)]{Buzzard} (when $p \mid N$), and \cite[p.9 before (3.5)]{KR08} (when $p \mid D$), the number of supersingular points on $\Curve{D}{N}_{\F_p}$ is $H_p(D,N)$,
    yielding the lower bound
    $$
    \frac{1}{2^{\omega(DN)-r}} H_p(D,N) \le \#X(\F_{p^2}).
    $$
    If $D = 1$, we can improve the lower bound by adding the cusps, as in \cite[Proposition 1]{FH99} 
    In both cases, this leads to a contradiction. 
\end{proof}

Let us denote
$$
    f_1(D) \colonequals \frac{1}{2^{\omega(D)}} \varphi(D), \quad
    f_2(N) \colonequals \frac{1}{2^{\omega(N)}} \psi(N), \quad 
    g(p) \colonequals \frac{24(p^2+1)}{p-1},
$$
and denote by $p_n$ the $n$-th prime, so that $p_1 = 2, p_2 = 3$ and so forth.

The following lemma is based on \cite[Lemma, p.181]{HH96}.

\begin{lemma} \label{lem: finite bounds}
  If $r+s \ge 7$, then 
  \begin{equation} \label{eq: finite bounds on prime products}
  f_1(p_1 p_2 \cdots p_r)f_2(p_{r+1} p_{r+2} \cdots p_{r+s}) > g(p_{r+s+1}).
  \end{equation}
\end{lemma}

\begin{proof}
   One checks that for $r = 7$ and $s = 0$ we have $f_1(p_1 p_2 \cdots p_r) = 720 > 1448/3 = g(p_{r+1})$. Moreover, since $f_1(p) < f_2(p)$ for all $p$, we have
   \begin{equation*}
   f_1(p_1 \cdots p_r) f_2(p_{r+1} \cdots p_{r+s}) \ge f_1(p_1 \cdots p_{r+s}),
   \end{equation*}
   showing that if $r+s = 7$, the result holds.
   Assume that $r+s \ge 8$ and $s \ge 1$. Then
    $$
    \frac{f_2(p_{r+1} \cdots p_{r+s})}{f_2(p_{r+1} \cdots p_{r+s-1})} = f_2(p_{r+s}) = \frac{1}{2}(p_{r+s} - 1) \ge 9
    $$
    and 
    $$
    \frac{g(p_{r+s+1})}{g(p_{r+s})} = \frac{p_{r+s+1}^2 + 1}{p_{r+s}^2+1} \cdot \frac{p_{r+s} - 1}{p_{r+s+1}-1}
    < \frac{p_{r+s+1}^2+1}{p_{r+s}^2+1} \le \frac{1 + 4p_{r+s}^2}{1 + p_{r+s}^2} < 4.
    $$
    Proceeding by induction on $s$, if \eqref{eq: finite bounds on prime products} holds for $r, s - 1$, then
    $$
    f_1(p_1\cdots p_r)f_2(p_{r+1} \cdots p_{r+s}) \ge 9 f_1(p_1 p_2 \cdots p_r)f_2(p_{r+1} p_{r+2} \cdots p_{r+s-1}) > 4 g(p_{r+s}) > g(p_{r+s+1}).
    $$
    Therefore, the claim holds for all $r,s$ such that $r+s \ge 7$.
\end{proof}

\begin{theorem}
\label{thm:finiteness}
    The set 
    $
    H \colonequals \{ (D,N) \in \N^2 : \gamma_{\Qbar}(\StarCurve{D}{N}) \le 2 \}
    $
    is finite.
\end{theorem}

\begin{proof}
     Let $C = 2^6 \cdot 435$ and $A = p_1 p_2 \cdots p_8 = 9699690$. We will show that if $$N \cdot \varphi(A) \left( \frac{D}{A} \right)^{\frac{\ln(18)}{\ln(19)}}\ge C$$ then $(D,N) \notin H$, hence the result.

     Write 
     $$
     D = \prod_{k=1}^r p_{i_k}^{\alpha_k}, \quad N = \prod_{l=1}^s p_{j_l}^{\beta_l}
     $$ 
     with $i_1 < i_2 < \ldots < i_r$, $j_1 < j_2 < \ldots < j_s$ and $\alpha_k, \beta_l\ge 1$ for all $k,l$. 

     Note that for any prime $p$ and an integer $\alpha \ge 1$, one has $f_1(p^{\alpha}) \ge f_1(p)$ and $f_2(p^{\alpha}) > f_2(p)$. Also, if $p < q$ are two primes, then $f_1(p) < f_1(q)$, $f_2(p) < f_2(q)$ and $f_1(p)f_2(q) < f_1(q) f_2(p)$.

     Then $f_1(D) \ge \prod_{k=1}^r f_1(p_{i_k})$ and $f_2(N) \ge \prod_{l=1}^s f_2(p_{j_l})$.
     Moreover, by this last property, if $\{i_k \}_{k=1}^r \cup \{j_l \}_{l=1}^s = \{m_t\}_{t=1}^{r+s}$, with $m_1 < m_2 < \ldots < m_{r+s}$, then 
     $$
     f_1(D) f_2(N) \ge \prod_{t=1}^r f_1(p_{m_t}) \prod_{t=r+1}^{r+s} f_2(p_{m_t})
     \ge  \prod_{t=1}^r f_1(p_t) \prod_{t=r+1}^{r+s} f_2(p_t).
     $$
     It follows that $f_1(D)f_2(N) \ge f_1(p_1 \cdots p_r) f_2(p_{r+1} \cdots p_{r+s})$.
     If $r+s \ge 7$, then Lemma~\ref{lem: finite bounds} yields $f_1(D) f_2(N) > g(p_{r+s+1})$, hence \eqref{eq: finite bounds} holds (with $r = 0$) for the smallest prime $p$ such that $p \nmid DN$, as $p \le p_{r+s+1}$, which by Proposition~\ref{prop: finite bounds} means that $(D,N) \notin H$. 

     On the other hand, if $r + s \le 6$ then $p \le p_7 = 17$, hence $g(p) \le g(17) = 435$. 
     By Landau's theorem (see e.g. \cite[proof of Theorem 328, \S 22.9]{HardyWright}), for $\delta = 1 - \log_{p_8}(p_8 - 1)$, the function $\varphi(n)/n^{1-\delta}$ attains its minimum at $A = p_1  p_2 \cdots p_8$. Therefore, as $p_8 = 19$, we get
     $$
     f_1(D) f_2(N) \ge \frac{\varphi(A)}{2^6} \left( \frac{D}{A} \right)^{\frac{\ln(18)}{\ln(19)}} \cdot N \ge \frac{C}{2^6} = 435,
     $$
     showing that \eqref{eq: finite bounds} holds in this case as well, hence $(D,N) \notin H$.
\end{proof}

Using the explicit bounds yields an effective method to enumerate this finite list of pairs, leaving us with $2342$ pairs $(D,N)$ for which \eqref{eq: finite bounds} does not hold. There are $424$ values of $D$ we have to check, with the maximal possible $D$ being $39270 = 2 \cdot 3 \cdot 5 \cdot 7 \cdot 11 \cdot 17$. 

\section{The genus formula and fixed points of Atkin--Lehner involutions}
\label{sec:genusfixedpoints}
We first recall the genus formula for  $\Curve{D}{N}$, and then, using fixed point formulas for the Atkin--Lehner involutions, give a genus formula for the quotient curve $\Curve{D}{N} / W$. 
We then use the fixed point formulas to give another test for subhyperellipticity.

Recall that  $B = B_D$ is the quaternion algebra of discriminant $D$ over $\Q$, and let $\Order \subseteq B$ be an Eichler order of level $N$.
Throughout the paper we use the notation $w_d $ for the unique element  in $\ALfull{D}{N}$ represented by some $\gamma_d \in N_{B^{\times}}(\Order)$ with $\nrd(\gamma_d) = d$. 

\subsection{The genus formula}
\label{sec:genus}
These formulas rely on counting fixed points which arise as CM points (or cusps) on $\Curve{D}{N}$. 
These fixed points can be counted by counting optimal embeddings of quadratic orders.

\begin{definition}
    Let $K/\Q$ be a separable quadratic and $R \subset K$ an order.  An embedding $\phi: R \to \Order$ is \emph{optimal} if 
    \[ \phi(K) \cap \Order = \phi(R) .\]

    For any $\gamma  \in \Order^\times$, the optimal embedding $\phi_\gamma (\alpha) \coloneqq \gamma \phi(\alpha) \gamma^{-1}$ is \emph{equivalent} to $\phi$.
    We write $\nu(R,\Order)$ to denote the number of equivalence (or conjugacy) classes of optimal embeddings from $R$ to $\Order$.
\end{definition}

Note that an embedding $\phi: R \hookrightarrow \Order$ is optimal if and only if the induced local embeddings $S_p \hookrightarrow \Order_p$ are optimal for all primes $p$. Let $\nu_p(R, \Order)$ denote the number of equivalence classes of optimal embeddings $R_p \hookrightarrow \Order_p$.

\begin{proposition}[Eichler, see {\cite[Proposition 1]{Ogg}}]
\label{prop:localembed}
Let $R$ be an order of conductor $f$ and denote by $h(R)$ the class number of $R$.
Then
\[ \nu(R,\Order) = h(R) \prod_{p} \nu_p(R, \Order).  \]
\end{proposition}

The local factors $\nu_p(R, \Order)$ have explicit formulae and can be computed easily, see  \cite[Theorem 2]{Ogg} for a full description.

Write $\nu_\infty(D,N)$ for the number of cusps on $\Curve{D}{N}$, which is $2^{\omega(N)}$ when $D  = 1$ and $0$ otherwise.
The genus of $\Curve{D}{N}$ is 
\[ g(D,N) =1 +  \frac{1}{12} \prod_{p | D} (p - 1) 
\prod_{p|N} (p+1)  -
\frac{\nu(\Z[\sqrt{-1}], \Order)}{4} - \frac{\nu(\Z[\zeta_3], \Order)}{3} - 
\frac{\nu_\infty(D,N)}{2}. \]

Let $m \mid DN$ and consider the Atkin--Lehner involution $w_m.$ To obtain a genus formula for $\Curve{D}{N}/W$ where $W\subset \ALfull{D}{N}$ we need to count fixed points of $w_m$. Let $\mu \in \Order$ denote an element of norm $m$ which induces the automorphism $w_m$ (under conjugation). To this element we can associate one or two equivalence classes of optimal embeddings.

\begin{theorem}[{\cite[(4)]{Ogg}}]
\label{thm:ALfixedpoints}
    The number of fixed points of $w_m$ is  (except for the case $D = 1$ and $m = 4$)
    \begin{equation} \label{eq:ALfixedpoints}
    f_m \coloneqq \sum_R h(R) \prod_{\substack{p \mid DN \\ p \nmid m}} \nu_p(R, \Order),
    \end{equation}

    where the sum is taken over:
    \begin{itemize}
       \item $\Z[\sqrt{-m}]$, $\Z[\frac{1+\sqrt{-m}}{2}]$ when $m \equiv 3 \bmod 4$,
        \item $\Z[\sqrt{-m}]$ when $m \not \equiv 3 \bmod 4$ and $m \ne 2$,
        \item  $\Z[\sqrt{-2}]$, $\Z[\sqrt{-1}]$ when $m=2$.
    \end{itemize}
\end{theorem}

The genus formula for the quotient follows from applying the Riemann--Hurwitz formula.
\begin{theorem}[Genus formula]
Let $ W$ be a subgroup of Atkin--Lehner involutions of size $2^r$.
    The genus of $\Curve{D}{N}/W$ is
    \[1 + 2^{-r}(g(D,N)-1) - 2^{-r-1}\sum_{w_m \in W\setminus \{1\}} f_m .\]
\end{theorem}
An immediate application of this formula is that all the genus  $0,1,$ and $2$ curves are subhyperelliptic and can be identified using the formula.

We find that, out of the $2342$ star curves that we begin with, there are $336$ curves $\StarCurve{D}{N}$ for which the genus is $0$, $341$ for which the genus is $1$, and $346$ for which the genus is $2$. These curves are thus all subhyperelliptic, and $1319$ star curves remain after filtering by genus.

\subsection{Fixed points of Atkin--Lehner involutions}
\label{sec:fixedpointsAL}
The following gives a basic subhyperellipticity test, using Theorem \ref{thm:ALfixedpoints} to count fixed points.
\begin{proposition}[{\cite[Proposition 1]{OggHyperelliptic}} ]
\label{prop:fixedpointsAL}
 Assume $X = \Curve{D}{N}/W$ is hyperelliptic. Then any non-hyperelliptic involution on $X$ has either $4$ fixed points or none if the genus of $X$ is odd and exactly $2$ fixed points if the genus of $X$ is even.
\end{proposition}
\begin{example}
On the genus 5 curve $\Curve{6}{23}$, the involution $w_{46}$ has 8 fixed points. Therefore $\Curve{6}{23}$ is not hyperelliptic.
\end{example}

This same proposition can be used to rule out subhyperellipticity in more complicated ways. The following proposition is a generalization of \cite[Proposition 6]{FH99}.

\begin{proposition}
\label{prop:complicatedAL}
Let $G \le \Aut(\Curve{D}{N})$ be an elementary $2$-group of order $2^\omega$. Let $X = \Curve{D}{N}/G$ be a curve of genus $\ge 3$. Suppose there exist $g_1, g_2 \in \Aut(X_0(D,N)) \setminus G$ such that $g_1 g_2 \in G$ and the numbers of fixed points of $g_1$ and $g_2$ are given by $2^{\omega}$ and $3 \cdot 2^{\omega}$, respectively. Assume further that $g_2 = w_{N_2}$ is an Atkin--Lehner involution, and that
\begin{enumerate}
\item Either one of the following holds. \label{item:N2 congruence}
\begin{enumerate}
    \item $N_2 \not\equiv 3 \pmod 4$, or
    \item $N_2 \equiv 3 \pmod 8$ and $2 \mid N$, or \label{item:3 mod 8}
    \item $N_2 \equiv 7 \pmod 8$ and $2 \mid D$. \label{item:7 mod 8}
\end{enumerate}
\item In addition, $3 \mid h(-4 N_2)$. \label{item:class number}
\end{enumerate}
Then $X$ is not hyperelliptic.
\end{proposition}

\begin{proof}
    The proof is along the lines of \cite[Proposition 6]{FH99}. By assumption \eqref{item:class number}, $N_2 \ge 5$. Therefore, by \eqref{eq:ALfixedpoints} we find that 
    $$
    f_{N_2} = \begin{cases}
        c_1 \cdot h(-4N_2) + c_2 \cdot h(-N_2) & N_2 \equiv 3 \bmod 4, \\
        c_1 \cdot h(-4N_2) & N_2 \equiv 1,2 \bmod 4.
    \end{cases}
    $$
    If we assume that $N_2 \equiv 3 \bmod 4$, then we are either in case \eqref{item:3 mod 8} or \eqref{item:7 mod 8}. However, in both cases we have $\nu_2 \left(\Z[\frac{1+\sqrt{-N_2}}{2}], \Order \right) = 0$, hence $c_2 = 0$.
    This means that in either case the set of fixed points of $w_{N_2}$ consists of points $P$ such that $3 \mid [\Q(P) : \Q]$, and by our assumption also $[\Q(P) : \Q]  \mid 3 \cdot 2^{\omega}$. Since $G$ is an elementary $2$-group acting freely on this set (thanks to our assumption on $g_1$), it contributes to three conjugate fixed points of $w_{N_2}$ on $X$. By Proposition~\ref{prop:fixedpointsAL}, $w_{N_2}$ has exactly four fixed points on $X$, and the remaining one must be rational. But this implies that $X$ is non-hyperelliptic by \cite[Proposition 6]{HasegawaStarcurves}.
\end{proof}

The immediate application is when $G \le \ALfull{D}{N}$ is an  Atkin--Lehner group, as in \cite{FH99}.

\begin{corollary}
    Let $W \le \ALfull{D}{N}$ be of size $2^{\omega(DN) - r}$.
    Assume there exist $N_1, N_2$ such that $w_{N_1}, w_{N_2} \notin W$ and $w_{N_1} w_{N_2} \in W$, with number of fixed points $2^{\omega(DN)-r}$ and $3 \cdot 2^{\omega(DN)-r}$, respectively. If $N_2$ satisfies the conditions \eqref{item:N2 congruence} and \eqref{item:class number} of Proposition~\ref{prop:complicatedAL}, then $X = \Curve{D}{N}/W$ is not hyperelliptic.
\end{corollary}

\begin{example}
We find all quotients $X$ of $\Curve{6}{169}$ by $W \subset \ALfull{6}{169}$ that satisfy the conditions of Proposition \ref{prop:complicatedAL}.

First note that $N_2 = 169$ satisfies  $169 \equiv 1 \pmod 4 $ and $h(-4 \cdot 169) = 6$ so $169$ satisfies (1) and (2). (In fact, this is the only divisor of $169$ satisfying these conditions.)
So we want to find $w_{169}$ with $12$ fixed points: this determines that $r = \omega(6 \cdot 169) - v_2(12) = 1$ and fixes the size of $W$ that we are looking for.

We consider all possible $W \subsetneq \ALfull{6}{169}$ such that the quotient is at least genus $3$,  $W$ does not contain $w_{N_2}$, and $\# W = 2^{\omega(6 \cdot 169) - 1}$. These are:
\begin{align*}
\langle w_2, w_3 \rangle,
\langle w_2, w_{507} \rangle, 
\langle w_3, w_{338} \rangle,
\langle w_6, w_{338} \rangle.
\end{align*}
For each $W$ and $w_{N_2}$, we look for all $N_1$ such that $w_{N_1} \notin W$ and $w_{N_1}$ has the number of fixed points equal to $\# W$. 
In the first case there is no $w_{N_1}$, so we eliminate this from consideration.
For the final three cases we find
\[ N_1 = 3, N_1 = 2,  N_1 \in \{2,3\}.\]
In each of these cases we find that $w_{N_1} w_{N_2} \in W$,
proving that in these three final cases, the quotient curve is non-hyperelliptic.
\end{example}

However, Proposition~\ref{prop:complicatedAL} can also be applied in the presence of modular non-Atkin--Lehner involutions, see Section~\ref{sec:modularnonAL} for their definition.

\begin{example}
    Consider $X = \Curve{10}{153} / \langle w_2, w_9, w_{85} \rangle$. Since $9 \parallel 153$, there is an involution $V_3 \in \Aut(\Curve{D}{N})$. Let $G = \langle w_2 , V_3, w_{85} \rangle$, then conjugation by $S_3$ induces an isomorphism $X \simeq \Curve{10}{153} / G$. On the latter curve, the number of fixed points of $w_{18}$ is $8$ and the number of fixed points of $w_{1530}$ is $24$, showing that $X$ is not hyperelliptic. 
\end{example}

\section{Trace formulas and point counting}
\label{sec:trace}
Let $X = \Curve{D}{N}/W$. 
Let $p$ be a prime number with $p \nmid DN$ and $v \in \N$. 
In this section we describe a technique for counting $X(\F_{p^v})$ using trace formulas. 

\subsection{Point counting}
Since $X$ has good reduction at $p$, generalizing \eqref{bound:finitefield}, if $X$ is subhyperelliptic, the point count over $\F_{p^v}$ satisfies 
\begin{equation}
\label{eq:small point count}
\#X(\F_{p^v}) \leq 2 (1+p^v).
\end{equation}
Let $g$ be the genus.
Then \cite[p.182]{HH96} show that \eqref{eq:small point count} is automatically satisfied when $p^v>4g^2$. Thus, checking point counts is only useful up to this bound.
To count $\#X(\F_{p^v})$, we use an isogeny, analogous to the one in \cite{Ribet}, but first introduce some notation.

For any $m \mid DN$, let $\sigma_m \colonequals (-1)^{\omega(\gcd(m,D))}$, and let
$$
S_2(DN; W = \sigma) \colonequals \{ f \in S_2(DN) : w_m f = \sigma_m f \quad \forall w_m \in W \}.
$$
For any $M \mid N$, and any $d | (N/M)$, let $\iota_d : S_2(DM) \to S_2(DN)$ be the corresponding level-raising map. 
Then 
$$
S_2(DN)^{\Dnew} = \bigoplus_{M \mid N} \bigoplus_{d \mid N/M} \iota_d(S_2(DM)^{\new}).
$$
For each $M \mid N$, let $S_{D,M}$ be a set of representatives for Galois orbits of newforms $f \in S_2(DM)^{\new}$. Then 
$$
S_2(DN)^{\Dnew} =  \bigoplus_{M \mid N} \bigoplus_{f \in S_{D,M}} \bigoplus_{d \mid N/M} \iota_d(V_f),
$$
where $V_f \subseteq S_2(DM)^{\Dnew}$ is the Galois module generated by $f$. 
Denote 
$$
S(W, M) \colonequals \{ (f,d) : f \in S_{D,M}, \quad d \mid N/M, \quad w_m (\iota_d(f)) = \sigma_m \iota_d(f) \},
$$
and $S(W) = \bigcup_{M \mid N} S(W,M)$, so that 
$$
S_2(DN; W = \sigma)^{\Dnew} = \bigoplus_{(f,d) \in S(W)} \iota_d(V_f).
$$
\begin{proposition} \label{prop:decomposition}
There is an isogeny
$$
J(\Curve{D}{N} / W) \sim \bigoplus_{(f,d) \in S(W)} A_f 
$$
defined over $\Q$, where $A_f$ is the abelian variety associated to $f$. 
\end{proposition}

\begin{proof}
    For $W = \{1 \}$, this is the Eichler--Shimura relation together with the Jacquet--Langlands correspondence as in \cite[Theorem 2.12]{Hida}, identifying the quaternionic modular forms of level $\Order$ with $S_2(DN)^{\Dnew}$, using a map $\JL : S(\Order) \to S_2(DN)^{\Dnew}$. For more details, see \cite[Proposition 1.18(iii)]{Nekovar}.

    For general $W$, we note that the action of $W$ on $\Curve{D}{N}$ induces an action of $W$ on the Jacobian, with $J(\Curve{D}{N}/W) \sim J(\Curve{D}{N})^W$. Furthermore, from \cite[Theorem 8.1]{GelbartJacquet} it follows that $w_m \JL(f) = \sigma_m \JL(w_m f)$ for all $w_m \in \ALfull{D}{N}$, hence the result.
\end{proof}

Writing 
\begin{equation} \label{eq:trace of Hecke}
a_{p^v} := \Tr \left( T_{p^v} | S_2(DN; W = \sigma)^{\Dnew} \right),
\end{equation}
where $T_n$ are the Hecke operators on spaces of modular forms, the proposition has the following immediate corollary.

\begin{corollary}
Let $X = \Curve{D}{N} / W$.
With $a_{p^v}$ defined as in \eqref{eq:trace of Hecke} and $a_{p^{-1}} = 0$, 
\begin{equation} \label{eq:point counts from traces}
    \# X(\F_{p^v}) = p^v + 1 - (a_{p^v} - p a_{p^{v-2}}).
\end{equation}
\end{corollary}

\begin{proof}
    By Proposition~\ref{prop:decomposition}, $a_{p^v} - p a_{p^{v-2}}$ is the trace of Frobenius on $J(\Curve{D}{N}/W)$ over $\F_{p^v}$. Finish by applying the Lefschetz trace formula.
\end{proof}

In order to compute $a_{p^v}$, we note that 
\begin{equation} \label{eq:trace on al fixed}
a_{p^v} = \Tr \left( T_{p^v} | S_2(DN; W = \sigma)^{\Dnew} \right)
= \frac{1}{\# W} \sum_{w_m \in W} \sigma_m  \Tr
\left( T_{p^v} \circ w_m| S_2(DN)^{\Dnew}  \right),
\end{equation}
and this last quantity can be computed using \cite[Corollary 5.5]{Assaf}, and a variation of \cite[Lemma 4.20]{Assaf} where one sums only over levels $DN'$ with $N' \mid N$. Note that since $p \nmid DN$, the formula simplifies to
$$
\Tr( T_{p^v} \circ w_m | S_2(DN)^{\Dnew} )
= \sum_{\substack{N' \mid N \\ \frac{m}{(m,DN')} = \square}}
\sigma_0 \left( \frac{N(m,DN')}{N'm} \right) \Tr( T_{p^v} \circ w_{(m,DN')} | S_2(DN')^{\new} ).
$$

Substituting it back into \eqref{eq:trace on al fixed} allows us to compute the point counts using \eqref{eq:point counts from traces}, and going back to \eqref{eq:small point count}, prove that when the inequality fails, the curve is not hyperelliptic.

Among the $1319$ star quotients $\StarCurve{D}{N}$ with genus $g \ge 3$, we find that $776$ are not hyperelliptic using this test. Among the $12502$ Atkin--Lehner quotients that satisfy \eqref{eq: finite bounds} and have genus $g \ge 3$, and were not found to be hyperelliptic in other means, we find that $1458$ are not hyperelliptic using this test.

\begin{example}
    Consider the curve $X = \StarCurve{6}{223}$. Counting points over $\F_5$, we find that $a_5 = -7$, hence $\# X(\F_5) = 13 > 12 = 2(1+5)$, contradicting \eqref{eq:small point count}. Therefore, it is not hyperelliptic. 
\end{example}

\subsection{Modular non-Atkin--Lehner involutions}
\label{sec:modularnonAL}
A main method for deciding that a curve $X$ \emph{is} hyperelliptic is to exhibit it as a degree two cover $X \to \P^1$.
When this arises from the quotient by an Atkin--Lehner involution, it is easy to identify, by computing the genus of the quotient curve.

For values of $N$ such that $4 \mid N$ or $9 \parallel N$, the curve $\Curve{D}{N}$ admits additional modular automorphisms, which we proceed to describe. 

Let $\iota_{\infty} : B \otimes \R \to M_2(\R)$ be a real embedding of $B$, and let $\Gamma_0(D,N) = \iota_{\infty}(\Order^1) / \{ \pm 1 \}$ be the image of the elements of norm $1$ in $\PSL_2(\R)$. 
For $\mu \in \{2, 3\}$, define $S_\mu = \begin{pmatrix} \mu & 1 \\ 0 & \mu \end{pmatrix} \in \PSL_2(\R)$.
When $4 \mid N$ or $9 \parallel N$ the matrices $S_2$ and $S_3$ belong to the normalizer of $\Gamma_0(D,N)$ in $\PSL_2(\R)$, giving extra elements of the normalizer to consider (see  \cite[Section 2]{FH99} for the case of modular curves).
By \cite[Lemma 1]{FH99} the matrices $S_{\mu}$ induce the following order 2 automorphisms.
\begin{itemize}
\item $S_2$ if $4 \parallel N$,
\item $V_2 = S_2 W_{2^{v_2(N)}} S_2$ if $8 \mid N$,
\item $V_3 = S_3 W_9 S_3^2$  if $9 \parallel N$.
\end{itemize}

Moreover, \cite[Lemma 1]{FH99} describes their commutation relations with elements of $W$. Thus, for any such $N$ these generate additional automorphisms of the curve. 
Suppose $X = \Curve{D}{N}/W$ has an extra automorphism $v$. The strategy we will employ is to compute the genus $g_v$ of $X / \langle v \rangle$. 
If $g_v = 0$, $X$ is hyperelliptic, since $v$ is a hyperelliptic involution. 
If $g_v = 2$ and $g(X) = 3$, then by Proposition \ref{prop: genus 3 covering genus 2} $X$ is hyperelliptic.
Otherwise, from Riemann--Hurwitz we deduce the number of fixed points $f_v = 2g(X) - 4g_v + 2$ of $v$ on $X$. Applying Proposition~\ref{prop:fixedpointsAL} with the involution $v$, if either $f_v \notin \{0, 4\}$ when $g(X)$ is odd, or $f_v \ne 2$ when $g(X)$ is even, we may deduce that $X$ is not hyperelliptic.

Finally, we note that $g_v = \dim S_2(DN; W = \sigma)^{\Dnew, v}$, where 
$$
S_2(DN; W = \sigma)^{\Dnew, v} = \{ f \in S_2(DN; W = \sigma)^{\Dnew} : v(f) = f \}
$$
is the subspace of modular forms that are fixed under $v$.
In order to compute this dimension, we can again apply a trace formula, only this time for the operator $v$, following \cite{PopaI}. Using the definitions, this is rather slow and can be improved considerably in a similar way to the ideas of Oesterl{\'e} for the Atkin--Lehner operators (see \cite[\S 4]{PopaII}). However, the spaces of modular forms for which we test this are small enough for a direct computation using modular symbols. For the reader interested in the trace formulas, see Appendix \ref{appendix}.

\begin{example}
    The curve $X = \StarCurve{10}{261}$ has a modular automorphism given by $V_3$ since $9 \parallel 261$. Since $g(X) = 4$ and $g(X/ \langle V_3 \rangle) = 1$, we deduce that $V_3$ has $6$ fixed points on $X$, hence by Proposition~\ref{prop:fixedpointsAL}, $X$ is not hyperelliptic.
\end{example}

\begin{example}
    The curve $X = \StarCurve{15}{88}$ has a modular automorphism given by $V_2$ since $8 \mid 88$. Since $g(X / \langle V_2 \rangle) = 0$, $X$ is hyperelliptic with $V_2$ being the hyperelliptic involution. 
\end{example}

Using this test we find $63$ hyperelliptic curves and prove that $54$ are not hyperelliptic.

\section{Covers}
\label{sec:covers}
The covering structure of the Shimura curves $\Curve{D}{N}/W$ plays an important role in determining and propagating the subhyperellipticity.
\subsection{Atkin--Lehner involutions as hyperelliptic involutions}
\label{sec:ALashyp}
The quotient $\Curve{D}{N}/W$ is hyperelliptic when one of its Atkin--Lehner involutions $w \in \ALfull{D}{N} / W$ acts like the hyperelliptic involution. This happens if and only if the quotient by $w$ is genus 0.
\begin{example}
\label{ex:coveringhypAL}
The curve $\Curve{6}{29}$ is genus 5. This is hyperelliptic, with hyperelliptic involution $w_{174}$, because the genus of the quotient curve is 0.
\end{example}

\subsection{Closure}
A basic, but important, example of leverage the covering structure is the closure principles, explained in Proposition \ref{def:closure}.

\begin{example}
The genus 5 curve $\Curve{6}{29}$ was shown to be hyperelliptic in Example \ref{ex:coveringhypAL}. As a result, applying downward closure to $\Curve{6}{29} \to \Curve{6}{29}/\langle w_3 \rangle $, we can conclude that this genus 3 curve is also hyperelliptic. 

The genus 3 curve $\Curve{6}{133}/\langle w_2, w_{5}, w_{57} \rangle$ is non-hyperelliptic because it has too many points over $\F_5$, violating the equation \eqref{eq:small point count}. 
Using upward closure, we conclude that the genus 5 curve $\Curve{6}{133}/\langle w_2 , w_{399} \rangle$ is not hyperelliptic.
\end{example}

\begin{remark}
In addition to the covering structure of the quotients $\Curve{D}{N}/W$ for $W \subset \ALfull{D}{N}$ a helpful perspective is to think of extending the full set of quotients to include all $\Curve{D}{N}/G$ for $G$ containing  Atkin--Lehner or modular non-Atkin--Lehner involutions as in Section \ref{sec:modularnonAL}. 
This perspective allows one to  
apply techniques that are only leveraging the automorphism group or covering structure to this larger group.
\end{remark}

\subsection{Genus 3 covering genus 2}
\label{sec:g3coverg2}

\begin{proposition} \label{prop: genus 3 covering genus 2}
Let $k$ be a field of characteristic not 2. Let $C$ and $D$ be nice curves of genus 3 and 2, respectively, over $k$. Suppose there is a dominant map $\pi: C \to D$ over $k$. Then $C$ is hyperelliptic.
\end{proposition}

This is mentioned in \cite[Proposition 5]{HasegawaStarcurves} without proof. For the interested reader, we provide two distinct proofs.
\begin{proof}
    We can assume that $k = \bar{k}$. Then $D$ is hyperelliptic and birationally equivalent to a curve of the form $y^2 = f(x)$, where $f$ is a degree 6 separable polynomial in $k[x]$.
    Applying Riemann--Hurwitz, we can compute that $\pi$ is a degree 2 \'etale map.
     We claim that degree 2 connected finite \'etale  covers arise as factorizations $f = gh$ into nonconstant even-degree polynomials. Note that by Abhyankar's lemma any such factorization defines a degree 2 connected finite \'etale cover of D with function field
     $k(x)(\sqrt{f},\sqrt{g}) = k(x)(\sqrt{g},\sqrt{h})$, and furthermore there are $2^4 -1 = \binom{6}{2}$ of these.
     So the function field $k(C)$ has this form, say with $\deg g = 2$ and $\deg h = 4$.  Then $k(C)$ is a degree 2 extension of $k(x)(\sqrt{g})$, which is the function field of a conic, hence a rational function field.  Therefore $C$ is hyperelliptic.
\end{proof}

\begin{proof}
    Suppose $ \pi: C \to D$ is unramified. Since $D$ is hyperelliptic, let $\iota \in \Aut(D)$ be the hyperelliptic involution. Then since $\pi$ is a map of covering spaces of the  $C$ and $D$, we have that $\iota$ lifts to an involution of $C$ (and an involution of the Riemann surface of $C$). By GAGA, this yields an algebraic involution $\tilde{\iota}$ on the algebraic curve $C$. Therefore $C$ is hyperelliptic.
\end{proof}

\begin{example}
The curve $\Curve{6}{11}$ is genus 3 and covers the genus 2 curve $\Curve{6}{11}/\langle w_{11} \rangle$. Therefore $\Curve{6}{11}$ is hyperelliptic.
\end{example}

\subsection{Degeneracy morphisms}\label{sec:degeneracy}
We can also use the covering structure of the curves in another way: specifically, we know that some of the covers come from degeneracy morphisms on the level structure.

Consider the curve $\Curve{D}{N}/W$, and recall that $\Curve{D}{N}$ is the moduli space of pairs $(A, \iota)$ where $\iota : \Order_0(N) \to \End(A)$ equips $A$ with multiplication by $\Order_0(N)$, the Eichler order of level $N$ inside $B = B_D$, the quaternion algebra of discriminant $D$. 

Let $M \mid N$ and $q \mid N / M$.
Let  $i^{(q)}_{N,M}:\Curve{D}{N} \to \Curve{D}{M}$  denote the degeneracy morphism
\begin{align*}
	i_{N,M}^{(q)}:\Curve{D}{N} &\to \Curve{D}{M} \\
    (A,\iota) &\mapsto (A, \iota \circ [q]),
	\end{align*}
where $[q] : \Order_0(M) \to \Order_0(N)$ is the conjugation map by an element $\alpha_q \in B$ which is trivial at primes dividing $D$, and $\diag(1,q)$ away from them.
    \begin{proposition}
Let $w_m \in \ALfull{D}{N}$. Assume that $q^2 \mid N$ and $M = N/q^2$. For any Atkin--Lehner involution $w_m$, the following square commutes.
	\[
	\xymatrix{\Curve{D}{N} \ar[r]^-{w_m^{(N)}} \ar[d]_{i_{N,M}^{(q)}} & \Curve{D}{N} \ar[d]^{i_{N,M}^{(q)}} \\
		\Curve{D}{M} \ar[r]^-{w_m^{(M)}} & \Curve{D}{M} }
	\]
        \end{proposition}

\begin{proof}
    This is the same argument as in \cite[Proposition 3.16]{HKLF}.
\end{proof}

\begin{corollary}
\label{ref:cordegeneracy}
Suppose $q^2 \mid N$ and $(N/q^2, q) = 1$. Let $W' \colonequals \{w_m : w_m \in W \mid (m,q) = 1 \}$.
    There is a morphism
    \[ \Curve{D}{N}/W \to \Curve{D}{N/q^2}/W'\]
    of degree $[\SL_2(\Z) : \Gamma_0(q^2)]/2$ if $w_{q^2} \in W$ and $[\SL_2(\Z) : \Gamma_0(q^2)]$ otherwise.
\end{corollary}

\begin{proof}
    We just need to show the degree calculation -- this follows from considering the commutative square
    	\[
	\xymatrix{\Curve{D}{N} \ar[r]^-{d} \ar[d]_{2^{\#W}} & \Curve{D}{N/q^2} \ar[d]^{2^{\#W'}} \\
		\Curve{D}{N}/W \ar[r]^-{d'} & \Curve{D}{N/q^2}/W' }
	\]
    where the labels denote the degrees of the morphisms.
    Then 
    $$
    d = [\Order_0(N/q^2) : \Order_0(N)] = [\Gamma_0(N/q^2) : \Gamma _0(N)] = [\SL_2(\Z) : \Gamma_0(q^2)],$$
    and we can solve for $d'$ accordingly.
\end{proof}

Concretely, we apply the corollary in the following way. We want to construct a map $X_0(N)/W \to \P^1$ of degree $d' = 3$. If $X_0(N)/W$ is genus at least 3, the Castelnuovo--Severi inequality  \cite[Proposition 2.1]{Poonen} implies that a curve cannot be both trigonal and hyperelliptic. We conclude $X_0(N)/W$ is not hyperelliptic.
To achieve this, we need $d' = 3$ and $X_0(N/q^2)/W'$ to have genus 0.
In particular, when $q^2 = 4$, then $[\SL_2(\Z):\Gamma_0(q^2)] = 6$, and this yields $d' =3$ in the case where $w_{4} \in W$. 
\begin{example}
 The morphism $\Curve{15}{28}/\langle w_3, w_4, w_{35} \rangle  \to \Curve{15}{7}/\langle w_3, w_{35} \rangle$ is degree 3, and $  \Curve{15}{7}/\langle w_3, w_{35} \rangle$ is genus 0. This proves that $\Curve{15}{28}/ \langle w_3, w_4, w_{35} \rangle $ is not hyperelliptic.
\end{example}

\section{Isomorphisms}
\label{sec:isomorphism}
\subsection{Extra modular automorphisms}

When $4 \parallel N$ or $9 \parallel N$ there are  some isomorphisms between the various quotient curves of $\Curve{D}{N}$, arising from the extra modular automorphisms $V_2$ and $V_3$ defined in Section~\ref{sec:modularnonAL}.

\begin{proposition}
    Let $N$ be such that either $4 \mid N$ or $9 \parallel N$, and let $\sigma \in \{ S_2, V_2, V_3 \}$ be such that $\sigma \Gamma_0(D,N) \sigma^{-1} = \Gamma_0(D,N)$, and $p = 2$ or $p = 3$ accordingly. Then, for any $W \le \ALfull{D}{N}$ we have
    $$
    \Curve{D}{N} / W \simeq \Curve{D}{N} / \langle \{ w_{p^v}^{\varepsilon(m)} w_m \}_{p \nmid m}, \sigma \rangle,
    $$
    where $v = v_p(N)$ and $\varepsilon(m) \in \{0,1\}$ with $\varepsilon(m) = 0$ if and only if either $p = 2$ or $m \equiv 1 \bmod 3$ or $9 \parallel m$ and $m/9 \equiv 1 \bmod 3$.
\end{proposition}

\begin{proof}
    The isomorphism is obtained by conjugation by $S_2 w_4 S_2, S_2$ or $S_3$, respectively.
\end{proof}

The above proposition is used to prove two propositions from \cite[Section 3]{FH99} which are readily adapted to handle the case of Shimura curve quotients.

\begin{proposition}
\label{prop:iso4}
    Let $N \geq 1$ be an integer such that $4 || N$ and write $N = 4N'$. Let $W = \{ w_4,  w_{m_1}, \dots, w_{m_r}\}$ be a subgroup of the Atkin--Lehner group of $\Curve{D}{N}$ containing $w_4$. We have the isomorphism
    \[\Curve{D}{4N'}/W 
\simeq \Curve{D}{2N'}/ \langle w_{m_1}, \dots, w_{m_r}\rangle. \]
\end{proposition}

\begin{proposition}
\label{prop:iso9}
Let $N \geq 1$ be an integer such that $9 \parallel N$.
Define $\varepsilon(M)$ to be
\[
\varepsilon(M) =
\begin{cases}
0 & \text{ if } M \equiv 1 \bmod 3 \text{ or if } 9 \parallel M \text{ and }  M/9  \equiv 1 \bmod 3 \\
1 & \text{ else.}
\end{cases}
\]
Let $W = \langle w_{N_1}, \dots , w_{N_s} \rangle \subseteq \ALfull{D}{N}$ and 
$W' = \langle w_9^{\varepsilon(N_1)} w_{N_1}, \dots , w_9^{\varepsilon(N_s)}w_{N_s}\rangle $.
Then \[ \Curve{D}{N}/W \simeq \Curve{D}{N}/W'.\]
\end{proposition}

\begin{example}
We have the isomorphism $\Curve{15}{14} / \langle w_3, w_{35} \rangle \simeq   \Curve{15}{28} / \langle w_{3}, w_4, w_5 \rangle $. The curve $\Curve{15}{28} / \langle w_{3}, w_4, w_5 \rangle  $ is not hyperelliptic because of a degeneracy morphism to $\Curve{15}{7}/ \langle w_3, w_5 \rangle$.

We have the isomorphism  $\Curve{10}{63}/\langle w_{5}, w_{63} \rangle \simeq \Curve{10}{63}/\langle w_7, w_{45} \rangle$. Then the curve $\Curve{10}{63}/\langle w_7, w_{45} \rangle$ is non-hyperelliptic by upward closure from $\Curve{10}{63}/ \langle w_7, w_{10}, w_{45} \rangle$. This latter curve is not hyperelliptic because $w_2$ has $6$ fixed points. 
\end{example}

We can also find isomorphism using the modular non-Atkin--Lehner involutions.

 \subsection{Special fiber isomorphisms}
 \label{sec:specialfiber}

We can also detect isomorphisms in characteristic $p$.
In \cite[Proposition~1]{HH96}, 
they note that if $X_0^*(pN)$ and $X_0^*(N)$ have the same genus, they will be isomorphic in characteristic $p$ for any prime coprime to $N$. 
The following generalizes this observation  to arbitrary genera.
 \begin{proposition}
 \label{prop:specialfiber}
 Let $D, N \geq 1$ be integers. 
Let $p$ be a prime number such that $p \nmid N$.
Let $W\subset \ALfull{D}{N}$ such that there exists some $w_m \in W$ with $p \mid m$.
Define $W' = \{ w_m : w_m \in W \mid p \notdiv m\}.$
Then there is a normalization $(\Curve{D}{N}/W')_{\F_p} \to (\Curve{D}{Np}/W)_{\F_p}$. 
\end{proposition}

\begin{proof}
By \cite[Theorem 4.7(v)]{Buzzard}, the special fiber $\Curve{D}{Np}_{\F_p}$ is composed of two irreducible components, both isomorphic to $\Curve{D}{N}_{\F_p}$, which cross transversely at the supersingular points. Since the $w_m$ for $m \mid D$ preserve the two components, and $w_p$ interchanges them,  
there is an isomorphism of the normalization of the special fiber $(\Curve{D}{Np}/W)_{\F_p}$ with the special fiber $(\Curve{D}{N}/W')_{\F_p}$.
\end{proof}

\begin{corollary}
\label{cor:specialfiber}
With the notations of Proposition~\ref{prop:specialfiber},
if $\Curve{D}{N}/W'$ is not hyperelliptic, then $\Curve{D}{Np}/W$ is not hyperelliptic.
\end{corollary}

\begin{example}
The curve $\Curve{10}{63}/ \langle w_{10}, w_{63} \rangle$ is isomorphic to $\Curve{10}{9}/ \langle w_{10} \rangle $ over $\F_7$. 
We show that the latter curve is not hyperelliptic by counting points over $\F_7$.
\end{example}

\section{Weil Polynomials}
\label{sec:Weilpolys}

\subsection{Weil polynomials over \texorpdfstring{$q \neq 2$}{q ne 2}}
Let $q$ be a prime of good reduction for a curve $X$. The Weil polynomial of a curve $X/\F_q$ encodes information about the point counts of $X/\F_q$ and the isogeny class of the curve. For small genus and primes, the LMFDB \cite{lmfdb} contains the exact list of which isogeny classes contain hyperelliptic Jacobians. 
We used this data, which is complete in $g = 3$, $p \leq 23$, $g = 4$, $p = 2, 3, 5$, and $g = 5,6$, $p = 2$, to rule out curves from being hyperelliptic. The Weil polynomials for Shimura quotients can be computed using the trace formula techniques described in Section \ref{sec:trace}.

Even without a census of hyperelliptic curves over $\F_q$ it is possible to give parity restrictions on the Weil polynomials of hyperelliptic curves. If the Weil polynomial of a Shimura curve quotient does not satisfy the  parity restrictions, then the curve is not hyperelliptic. 

On a hyperelliptic curve over $\F_q$, Frobenius acts on the Weierstrass points and forms a partition $\{d_i\}$ of $2g+2$ consisting of the cardinalities of the orbits. 
The following is \cite[Proposition 2.4]{CostaWeilPolys} in the case of characteristic $p \neq 2$ and extends to $p = 2$ readily, cf. \cite[Remark~2.6]{CostaWeilPolys}.
\begin{proposition}{\cite[Proposition 2.4]{CostaWeilPolys}. }
Let $C$ be a hyperelliptic curve of genus $g$ over $\F_q$ and $q= p^k$, $k \in \N$.
\begin{enumerate}
\item When $p\neq 2$, let $\{d_i\}_{i = 1}^r$ be a partition of $2g+2$ which records the sizes of the orbits of Frobenius acting on the $2g+2$ geometric Weierstrass points. 

\item When $p = 2$,  let the  $2$-rank of $C$ be $f$. Let $\{d_i\}_{i=1}^r$ be a partition of $f+1$ which records the sizes of the orbits of Frobenius acting on the $f + 1$ geometric Weierstrass points (by the Deuring--Shafarevich formula).
    \end{enumerate}
Then, in both cases, for any $\ell \neq p$  we have
    \begin{equation}
    \label{eq:admissibleweil}
    \det (1 - \Frob t | H^1_{\text{{\'e}t}}(C^{\alg}, \Q_{\ell})) \equiv 
    \prod_{i=1}^r (t^{d_i} - 1) / (t - 1) \pmod 2.
    \end{equation}
\end{proposition}

\begin{proof}
    When $p \neq 2$ this is \cite[Proposition 2.4]{CostaWeilPolys}. 
    For $p = 2$ the proof is similar to the original proof. 
    We only need to note that, as in \cite[Lemma 2.2]{CostaWeilPolys}, if $C\to \P^1$ is ramified (resp. split) at infinity, with fiber $\infty$ (resp. $\{ \infty_1, \infty_2 \}$), and its set of ramification (Weierstrass) points is $\calR$, then $J[2]^{\text{{\'e}t}}$  is generated by 
    $$
    \{ e_U : U \subseteq \calR \mid \#U \text{ is even} \},
    $$
    where
    \[ e_U = \sum_{P \in U} P - |U| (\infty),  \quad \text{ resp. }  \quad e_U = \sum_{P \in U} P - \frac{|U|}{2} (\infty_1+ \infty_2)\]
    In the case $p = 2$ the group can be expressed as the vector space obtained from $\F_2^{f+1}$ by considering all vectors with even number of non-zero entries, where rank shows there are no non-trivial relations. 
\end{proof}

Such polynomials of the form in equation \eqref{eq:admissibleweil} are called \defi{admissible Weil polynomials}. 
\begin{corollary}
The set of admissible Weil polynomials modulo $2$ is independent of the characteristic $p$ and depends only on $g$ when $p>2$ or  $f$ when $p = 2$.
\end{corollary}

In  \cite[Proposition 2.10]{CostaWeilPolys} it is shown that they are enumerated by partitions with distinct parts. A slight modification of the proof shows the following proposition.
\begin{proposition}[{\cite[Proposition 2.10]{CostaWeilPolys}}]
\label{prop:weilparity}
Let $q = p^k$ be a prime power.
Let $\calP$ be the set of all partitions of $2g+2$ into odd parts (or $f+1$ in the case of $p = 2$).
The admissible Weil polynomials modulo 2 for Jacobians of hyperelliptic curves of genus $g$ (resp. $2$-rank $f$) over a finite field $\F_q$ are indexed by $\calP$. In other words, each $\{d_i\}_{i = 1}^r \in \calP$  gives a unique admissible Weil polynomial mod 2
\[\prod_{i=1}^r (t^{d_i} - 1) / (t - 1) \mod 2 \]
and these are all of the admissible Weil polynomials mod 2.
\end{proposition}

Note that if $\prod_{j=1}^s (t^{e_j} - 1) / (t - 1) \mod 2$
is an admissible Weil polynomial, it is not necessarily the case that $\{e_j\}_{j=1}^s \in \calP$. However, there is some $\{d_i\}_{i = 1}^r \in \calP$ such that
\[\prod_{j=1}^s (t^{e_j} - 1) / (t - 1)\equiv\prod_{i=1}^r (t^{d_i} - 1) / (t - 1)  \mod 2  \]
\begin{proof}[Proof of Proposition \ref{prop:weilparity}]
The statement is essentially contained in the proof of \cite[Proposition 2.10]{CostaWeilPolys}. Note that every $\{d_i\}_{i = 1}^r$ partition of $2g+2$ (resp. $f+1$) with odd parts is equivalent to a partition with distinct parts by the same proof therein.
\end{proof}

\begin{remark}
In  \cite[Remark 2.2]{GonzalezAut}, Gonz\'alez  describes a technique using mod $2$ point counts to rule out the existence of involutions in the automorphism group of a curve. 
In the case of involutions, all the information gained by this  approach is already contained in the Weil polynomials, hence we do not gain anything by applying it (see \cite[Section 4.1]{CostaWeilPolys}).
\end{remark}

\begin{example}
The genus 3 curve $\Curve{6}{25}/\langle w_{50} \rangle$ over $\F_{11}$ has Weil polynomial
\[t^6 - 2t^5 + 33t^4 - 44t^3 + 363t^2 - 242t + 1331.\]
While this does satisfy the parity constraints from Proposition \ref{prop:weilparity}, we have data from the LMFDB in this case that completely classifies the hyperelliptic isogeny classes. The isogeny class this Weil polynomial represents does not contain a  hyperelliptic curve over $\F_{11}$.

The genus 5 curve $\Curve{6}{157}/\langle w_3, w_{157} \rangle $ has bad reduction at $p = 2,3$. At $p = 5,7, 11,$ $13, 17,$ and $19$ the Weil polynomials mod $2$ are one of the 15 Weil polynomials satisfying the parity constraints of Proposition \ref{prop:weilparity}. At $p = 23$, however, the Weil  polynomial mod $2$ is
\[t^{10} + t^9 + t^8 + t^6 + t^5 + t^4 + t^2 + t + 1 \pmod 2\]
which does not satisfy the parity constraints. This proves the curve is not hyperelliptic.
\end{example}

\if{false}
\subsection{Ramification Points}\label{sec:ramification}
Let $X$ be a smooth algebraic curve of genus $g$. If $X$ has an involution $\iota$ then, by Riemann--Hurwitz, the quotient map $X \to X/\iota$ can have at most $2g+2$ ramification points. In this section, we apply this fact to   $X = \StarCurve{D}{N}$  a star curve of genus $g > 2$. 
We explain a technique for showing the number of ramification points of the map $X \to X/\iota$  for any involution $\iota$ would exceed $2g+2$. This technique works over $\F_p$, where $p$ is a prime of good reduction for $X$. Then  $\Aut_\Q(X) \hookrightarrow \Aut_{\F_p}(X_{\F_p})$, so the non-existence of involutions in $\Aut_{\F_p}(X_{\F_p})$ implies the non-existence over $\Q$.
The following criterion is  \cite[Remark 2.2]{GonzalezAut}.
\begin{theorem}
\label{thm:fpauts}
    Let $X$ be a nice curve over $\F_p$ of genus $g>2$. 
    For $x \in \Z$, write $\mathrm{mod_2}(x)$ for the integer remainder when $x$ is divided by $2$.
    Consider the sequence
    \[ P_p(n) \colonequals  \mathrm{mod}_2 ( \sum_{d|n}  \mu(n/d) | X(\F_{p^d})|) .\]
    If there is some $k > 0 $ such that 
    \[  \sum_{j \geq 0}^k (2j+1) P_p(2j+1) > 2g +2,\] then $\Aut_{\F_p}(X)$ does not contain an involution.
\end{theorem}
\begin{proof}
Suppose there were some involution $\iota : X \to X$. Let $\calR$ denote the set of ramification points of the map $X \to X/\iota$ over $\overline{\F}_p$.  At each step $n$, we would like to count the number of ramification points defined over $\F_{p^n}$ (but we will not be able to do exactly that).

 Let us write
\[  X(\F_{p^n})' \colonequals | X(\F_{p^n}) \setminus   \cup_{\substack{d |n \\ d \neq n}} X(\F_{p^{d}}) |\]
for the set of points that are defined over $X(\F_{q^n})$ but not a smaller field.
Then
\[  \sum_{n \geq 1} |\calR \cap  X(\F_p^n)'| = |\calR|  \] and
by Riemann--Hurwitz, we have that $|\calR| \leq 2g+2.$

Consider the orbit of a point $Q \in X(\overline{\F}_{p^n})'$ under the action of $\iota$. If $Q \in \calR$, the orbit has order $1$, otherwise the orbit has order 2. Therefore $\mathrm{mod}_2 |X(\F_{p^n})' | = \mathrm{mod}_2|X(\F_{p^n})' \cap \calR | .$
Note that each time we have a fixed point of $\iota$ in $ X(\F_{p^n})'$, by applying the action of $\Gal(\F_{p^n} /\F_p)$, we obtain $n$ fixed points. Considering this Galois action (which stabilizes $ X(\F_{p^n})'$ and $\calR$), we see that $| X(\F_{p^n})'|$ and $| X(\F_{p^n})' \cap \calR|$ are divisible by $n$. So, for $n$ even we will not receive any information from our point counts, but for odd $n = 2j+1$ we have at least one ramification point, and therefore at least $n$, and so
 \[ n \cdot \mathrm{mod}_2(X(\F_{p^n})') \leq| X(\F_{p^n})' \cap \calR| .\]
Therefore  $X$ has no involution if it satisfies the inequality
\[  \sum_{j \geq 0} (2j+1) \cdot  \mathrm{mod}_2(X(\F_{p^{2j+1}})') < 2g+2 .\]
Now apply M\"obius inversion to count points in $X(\F_{p^{2j+1}})'$ (i.e. to get the formula for $P_p(n)$) and deduce the theorem above. 
\end{proof}
Note that it is enough to sum up to $k = 2g+2$, because of the bound on $\calR$.

\begin{example}
By applying Theorem \ref{thm:fpauts}, we prove that the genus 3 curve $\StarCurve{6}{317}$ has no $\F_7$-involutions and is therefore non-hyperelliptic.
\end{example}

\fi
\section{Results}
\label{sec:results}
\subsection{Filtering}
The algorithm proceeds in two phases.  
First, we enumerate and filter the star curves $\StarCurve{D}{N}$ (Table \ref{tab:star-curve-filtering}).  
Second, for each pair $(D,N)$, we expand to all Atkin--Lehner quotients $\Curve{D}{N}/W$ above any possibly subhyperelliptic star curve. We then apply the remaining filters to the full set of quotients (Table \ref{tab:quotient-filtering}).  
In the tables, ``new filtered'' means newly ruled out from being hyperelliptic at that stage, ``new proved'' means newly proved to be subhyperelliptic at that stage, while ``open'' is the number of curves  still not determined, and ``filtered'' and ``proved'' columns count the running totals in each category so far.
\begin{table}[ht]
\centering
\small
 \renewcommand{\arraystretch}{1.15}
  \setlength{\tabcolsep}{3pt}
\begin{tabularx}{\textwidth}{Xrrrrr}
\hline
Stage & Open & Filt. & Proved & New filt. & New proved \\
\hline
Reduction to finitely many pairs $(D,N)$ (Section~\ref{sec:reduction})
& 2342 & 0 & 0 & -- & -- \\
Genus $\leq 2$ curves (Section~\ref{sec:genus})
& 1319 & 0 & 1023 & 0 & 1023 \\
Finite field point count (Section~\ref{sec:trace})
& 504 & 815 & 1023 & 815 & 0 \\
Special fiber isomorphisms (Section~\ref{sec:specialfiber})
& 495 & 824 & 1023 & 9 & 0 \\
Weil polynomials (Section~\ref{sec:Weilpolys})
& 234 & 1085 & 1023 & 261 & 0 \\
Modular non-Atkin--Lehner involutions (Section~\ref{sec:modularnonAL})
& 194 & 1110 & 1038 & 25 & 15 \\
\hline
\end{tabularx}
\caption{Filtering the $2342$ star curves $\StarCurve{D}{N}$.}
\label{tab:star-curve-filtering}
\end{table}

At the end of the star curve phase, among the $2342$ star curves, $1038$ are
proved subhyperelliptic, $1110$ are ruled out, and $194$ remain open.  

For each pair $(D,N)$, we next generate all quotients $\Curve{D}{N}/W$, with $W\subset \ALfull{D}{N}$.
This expands the list from $2342$ star curves to $18379$ Atkin--Lehner quotients. At this initial point, we do yet not apply the closure principles (Proposition \ref{def:closure}), so the only filtered and proved curves are the star curves.
We then apply the remaining filters to the full list of Atkin--Lehner quotients. After each stage, we apply upward closure, downward closure, and isomorphism testing to propagate determinations through the cover lattice (Proposition~\ref{def:closure} and Section~\ref{sec:isomorphism}). 
\begin{table}[ht]
\centering \small 
 \renewcommand{\arraystretch}{1.15}
  \setlength{\tabcolsep}{3pt}
\begin{tabularx}{\textwidth}{Xrrrrr}
\toprule
Operation & Open & Filt. & Proved & New filt. & New proved \\
\midrule
Generate all quotients $\Curve{D}{N}/W$
& 16231 & 1110 & 1038 & -- & -- \\
Genus $\leq 2$ curves (Section~\ref{sec:genus})
& 13106 & 1110 & 4163 & 0 & 3125 \\
Atkin--Lehner involutions as hyperelliptic involutions (Section~\ref{sec:ALashyp})
& 11458 & 1110 & 5811 & 0 & 1648 \\
Propagate closure and isomorphism
& 10906 & 1661 & 5812 & 551 & 1 \\
Atkin--Lehner fixed points (Section~\ref{sec:fixedpointsAL})
& 2125 & 10442 & 5812 & 8781 & 0 \\
Propagate closure and isomorphism
& 2067 & 10500 & 5812 & 58 & 0 \\
Genus $3$ covers of genus $2$ curves (Section~\ref{sec:g3coverg2})
& 2051 & 10500 & 5828 & 0 & 16 \\
Propagate closure and isomorphism
& 2051 & 10500 & 5828 & 0 & 0 \\
Degeneracy morphism (Section~\ref{sec:degeneracy})
& 2032 & 10519 & 5828 & 19 & 0 \\
Propagate closure and isomorphism
& 2002 & 10549 & 5828 & 30 & 0 \\
Refined Atkin--Lehner fixed points (Section~\ref{sec:fixedpointsAL})
& 1629 & 10922 & 5828 & 373 & 0 \\
Propagate closure and isomorphism
& 1613 & 10938 & 5828 & 16 & 0 \\
Finite field point count (Section~\ref{sec:trace})
& 1053 & 11498 & 5828 & 560 & 0 \\
Propagate closure and isomorphism
& 1016 & 11535 & 5828 & 37 & 0 \\
Weil polynomials (Section~\ref{sec:Weilpolys})
& 892 & 11659 & 5828 & 124 & 0 \\
Propagate closure and isomorphism
& 881 & 11670 & 5828 & 11 & 0 \\
Modular non-Atkin--Lehner involutions (Section~\ref{sec:modularnonAL})
& 828 & 11705 & 5846 & 35 & 18 \\
Propagate closure and isomorphism
& 827 & 11706 & 5846 & 1 & 0 \\
Resolve classical modular curve cases from literature
& 784 & 11749 & 5846 & 43 & 0 \\
\hline
\end{tabularx}
\caption{Filtering all $18379$ Atkin--Lehner quotients after expansion.}
\label{tab:quotient-filtering}
\end{table}

\subsection{Discussion}
Of the 18379 candidate curves, we prove that 11749 are not hyperelliptic, 5846  curves are subhyperelliptic (of which 3423 are hyperelliptic).
For the remaining 784 candidate curves we remain unsure whether they are hyperelliptic. 
The genera of these 784 curves are (in format $g^e$ where $g$ is the genus and  $e$ is the multiplicity)
\[\{ 3^{519}, 4^{87}, 5^{163}, 6^{10}, 7^5 \}. \]
The number of non-trivial independent Atkin--Lehner involutions still acting non-trivially on the quotient $2^{\omega(DN)}/|W| - 1$ over all unknown  $\Curve{D}{N}/W$ (with multiplicity) are
\[  \{ 0^{157},\ 1^{464},\ 3^{163} \} \]
which shows that 157 are star curves (and thus minimal with respect to covering) and the remainder are also low in the covering structure. 
In fact, only $32$ of them cover another undecided curve, so their hyperellipticity largely needs to be solved individually (without relying on closure principles). 
An additional $38$ are isomorphic over $\F_p$ to another undecided curve and $21$ more are fully isomorphic to another curve which is undecided.  
In total there are $693$ curves that are isolated under the relations described in the paper (do not get resolved by resolving another curve).

For curves that were proven to be subhyperelliptic, the genera  are
\[\{0^{905},\;
1^{1518},\;
2^{1725},\;
3^{870},\;
4^{373},\;
5^{260},\;
6^{86},\;
7^{73},\;
8^{18},\;
9^{13},\;
10^{4},\;
11^{1}\}.\] 

In particular, we find that the genus $11$ curve $\Curve{770}{1} / \langle w_{10} \rangle$ is hyperelliptic, since  $\Curve{770}{1} / \langle w_{10}, w_{77} \rangle$ has genus $0$. 

In Figure \ref{fig:genus-distribution}, one can see the genus distribution of all curves under consideration and their status. In Figure \ref{fig:test-attribution} we display the test used to prove each of the 17595 resolved cases. While these tests depend on the context and order in which they are applied, it gives a sense of how effective each test is. 
For example, the one curve proved subhyperelliptic by isomorphism is $X_0(1,126)/\langle w_7, w_9 \rangle $, which is isomorphic to $X_0(1,252)/ \langle w_4, w_7, w_9 \rangle$ by Proposition \ref{prop:iso4}. We can also identify that  $V_3$ is its hyperelliptic involution, but the modular non-Atkin--Lehner involution test was run later. 

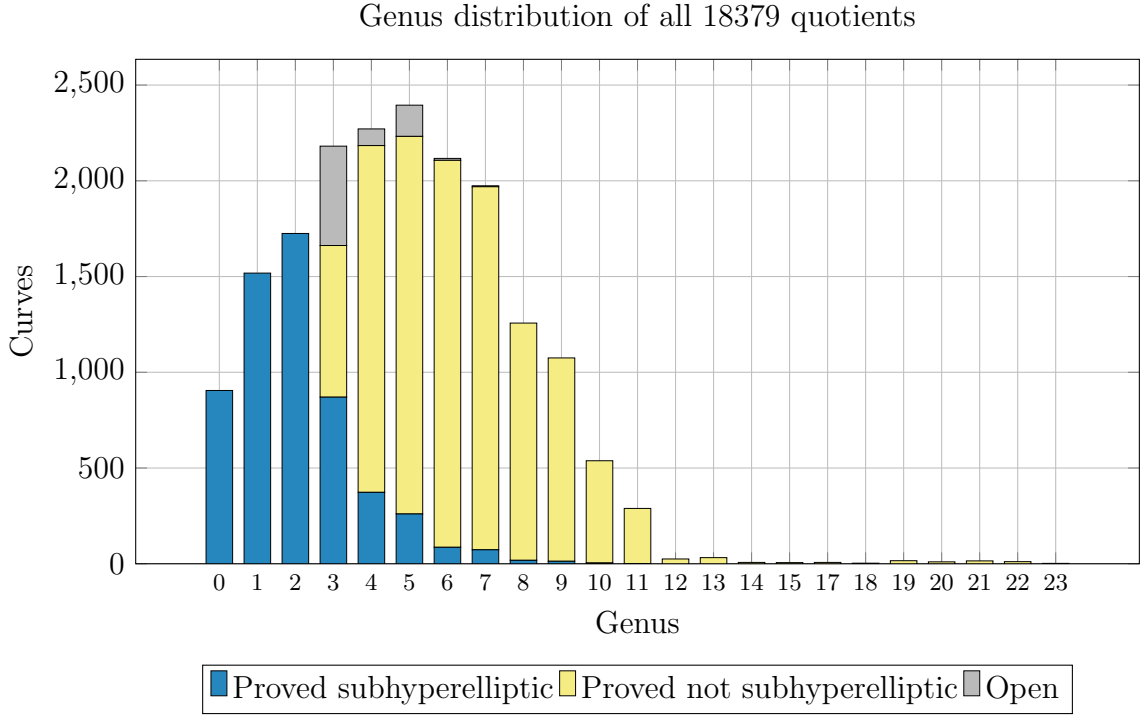
\begin{figure}[ht]
  \centering
  \begin{tikzpicture}
  \begin{axis}[
  ybar stacked,
  width=0.9\textwidth,
  height=0.50\textwidth,
  ymin=0,
  xlabel={Genus},
  ylabel={Curves},
  symbolic x coords={0,1,2,3,4,5,6,7,8,9,10,11,12,13,14,15,17,18,19,20,21,22,23},
  xtick=data,
  x tick label style={font=\scriptsize},
  legend style={at={(0.5,-0.20)}, anchor=north, legend columns=3},
  title={Genus distribution of all $18379$ quotients},
  grid=major,
  ]

  \addplot+[
    draw=black,
    line width=0.25pt,
    fill=StatusSub!85
  ] coordinates {
  (0,905) (1,1518) (2,1725) (3,870) (4,373) (5,260) (6,86) (7,73) (8,18) (9,13) (10,4) (11,1) (12,0) (13,0) (14,0) (15,0) (17,0) (18,0) (19,0) (20,0) (21,0) (22,0) (23,0) 
  };

  \addplot+[
    draw=black,
    line width=0.25pt,
    fill=StatusNot!65
  ] coordinates {
  (0,0) (1,0) (2,0) (3,792) (4,1811) (5,1972) (6,2021) (7,1896) (8,1239) (9,1062) (10,534) (11,288) (12,25) (13,32) (14,7) (15,6) (17,7) (18,3) (19,16) (20,10) (21,15) (22,11) (23,2) 
  };

  \addplot+[
    draw=black,
    line width=0.25pt,
    fill=StatusOpen!45
  ] coordinates {
  (0,0) (1,0) (2,0) (3,519) (4,87) (5,163) (6,10) (7,5) (8,0) (9,0) (10,0) (11,0) (12,0) (13,0) (14,0) (15,0) (17,0) (18,0) (19,0) (20,0) (21,0) (22,0) (23,0) 
  };

  \legend{Proved subhyperelliptic, Proved not subhyperelliptic, Open}
  \end{axis}
  \end{tikzpicture}
  \caption{Final genus distribution of all quotients, separated by status.}
  \label{fig:genus-distribution}
\end{figure}

\begin{figure}[ht]
\centering
\small
\renewcommand{\arraystretch}{1.12}

\begin{tabularx}{0.8\linewidth}{%
  >{\raggedright\arraybackslash}p{0.47\linewidth}
  r
  >{\raggedright\arraybackslash}p{0.26\linewidth}
}
\toprule
\textbf{Test} & \textbf{Count} & \textbf{Relative size} \\
\midrule
\multicolumn{3}{l}{\textbf{Proved subhyperelliptic}} \\
\addlinespace[0.2em]

Hyperelliptic Atkin--Lehner involution
  & 4466 & \testbar{StatusSub}{4466} \\

Genus $\leq 2$
  & 1330 & \testbar{StatusSub}{1330} \\

Modular non-Atkin--Lehner involution
  & 33 & \testbar{StatusSub}{33} \\

Genus 3 cover of genus 2
  & 16 & \testbar{StatusSub}{16} \\

Isomorphism
  & 1 & \testbar{StatusSub}{1} \\

\addlinespace[0.6em]
\midrule
\multicolumn{3}{l}{\textbf{Proved not subhyperelliptic}} \\
\addlinespace[0.2em]

Atkin--Lehner fixed points
  & 8781 & \testbar{StatusNot}{8781} \\

Finite field point count
  & 1375 & \testbar{StatusNot}{1375} \\

Upward closure
  & 496 & \testbar{StatusNot}{496} \\

Weil polynomial
  & 385 & \testbar{StatusNot}{385} \\

Refined Atkin--Lehner fixed points
  & 373 & \testbar{StatusNot}{373} \\

Isomorphism
  & 109 & \testbar{StatusNot}{109} \\

Special fiber isomorphism
  & 108 & \testbar{StatusNot}{108} \\

Modular non-Atkin--Lehner involution
  & 60 & \testbar{StatusNot}{60} \\

Classical modular curve (literature)
  & 43 & \testbar{StatusNot}{43} \\

Degeneracy morphism
  & 19 & \testbar{StatusNot}{19} \\

\bottomrule
\end{tabularx}

\caption{Distribution of the tests that determined the $17595$ resolved quotients. Bar lengths are log-scaled.}
\label{fig:test-attribution}
\end{figure}

\subsection{Prime bounds}
Some of the techniques above are prime-dependent. It is likely that pushing the prime bounds, for example, by computing more primes on Weil polynomials, would rule out a handful of extra curves. 
However, these computations are expensive, and unlikely to yield large gains. It is possible that many of the remaining open curves are genuinely in the isogeny class of a hyperelliptic curve, in which case methods using finite field point counts cannot rule them out.
The most expensive part of the Weil polynomials is the class number computation from point counting: we have cached class numbers for quadratic fields with discriminant bounded by $2^{40}$ from the \href{https://www.lmfdb.org/NumberField/QuadraticImaginaryClassGroups}{LMFDB} \cite{lmfdb} already.

\subsection{Special fiber method}
The isomorphism described in Proposition~\ref{prop:specialfiber} can also be used to rule out hyperellipticity in the following way, as in \cite[\S 5]{FH99}. If $X = \Curve{D}{Np}/W$ is hyperelliptic, then the hyperelliptic involution $\iota$ acts on the normalization of its special fiber, isomorphic to the special fiber of $X' = \Curve{D}{N}/W'$. Thus, there would be an element $\iota \in \Aut_{\F_p}(X')$ such that 
\begin{equation} \label{eq:special fiber case 2}
\iota(\alpha) = w_m \alpha
\end{equation}
for all properly $\F_{p^2}$-rational supersingular points $\alpha \in X_{\F_p}$, where $w_m \in W/W' \simeq \Z / 2\Z$ is a non-trivial representative. 

In cases where $\StarCurve{D}{1} \simeq \P^1$, the natural covering map $\Curve{D}{Np} \to \StarCurve{D}{1}$ allows us to calculate the supersingular points of $\Curve{D}{Np}_{\F_p}$. When  $g(X') = 0$, one can check whether there exists an element $\iota \in \Aut_{\F_p}(X')$ satisfying \eqref{eq:special fiber case 2}, and if it does not exist, $X$ is not hyperelliptic.

Similarly, when $W$ contains some $w_m$ such that $p \mid m$, the special fiber consists of two connected components, with $\iota w_p$ fixing each of them, and acting as Frobenius on the supersingular points. This could be used in a similar way to rule out hyperellipticity.

\begin{example}
Consider the curve $X = \Curve{6}{23} / \langle w_{23} \rangle$. By Proposition~\ref{prop:specialfiber}, the normalization of $X_{\F_{23}}$ is isomorphic to $\Curve{6}{1}_{\F_{23}}$. By \cite{BabaGranath}, a model of $X' = \Curve{6}{1}$ is given by the conic $x^2 + 3y^2 + z^2 = 0$ in $\P^2$. Since $(6:16:1) \in X'_{\F_{23}}$, we have $X'_{\F_{23}} \simeq \P^1_{\F_{23}}$ and $\Aut_{\F_{23}}(X') \simeq \PGL_2(\F_{23})$. The supersingular points in these coordinates are given by $(0:\pm \sqrt{-8}:1), (\pm \sqrt{-1} : 0 : 1), (\pm \sqrt{-3} : 1 : 0)$, and one verifies that no element of $\PGL_2(\F_{23})$ acts as Frobenius on these points. Therefore $X$ is not hyperelliptic. 
\end{example}

\subsection{Ceresa cycles}
We now discuss some negative results about Ceresa cycles and their shadows on Shimura curves, inspired by \cite{Ceresa} in which Lupoian and Rawson study shadows of Ceresa cycles on the modular curves $X_0(N)$. We hope these results shed some light on the situation for Shimura curves and their quotients and aid researchers in the further study.

Let $J$ be the Jacobian of $X$ and
let $\phi: J \to J$ be an endomorphism of $J$. Let $e \in X(\Q)$ be a basepoint of $X$ inducing an Abel--Jacobi map $i_e$. The Ceresa cycle associated to $(X,e)$ is the element \[[i_e(C) ] - [-1]^* [i_e(C)]  \] in $\mathrm{CH}^1(J)$.
If a curve $X$ is hyperelliptic, then the Ceresa cycle is trivial when $e$ is a Weierstrass point (in other words, $e = K_C/(2g-2)$).
\begin{definition}
The \emph{shadow} of $\phi$ is the degree 0 divisor
\[ \mathrm{Sh}(\phi) = (2g-2) F_\phi - \deg(F_\phi) K_X - \phi(K_X) - \phi^\vee(K_X) + (\deg(\phi) + \deg(\phi^\vee))(K_X)\]
where $F_\phi$ denotes the fixed points of $\phi$.
\end{definition}

In \cite[Proposition~4]{Ceresa} Lupoian and Rawson prove that if $[\mathrm{Sh}(\phi)] \in J$ is non-torsion then Ceresa is non-zero (and therefore that $X$ is not hyperelliptic).
In particular, they study the cases where $\phi = T_\ell$ or $\phi = w_d$ is a Hecke operator or an Atkin--Lehner involution.

Most of the remaining open cases in our work are genus $3$ curves. 
These curves do not cover genus 2 curves, or by  Proposition \ref{prop: genus 3 covering genus 2} they would already be known to be hyperelliptic. In this case, we show as a consequence of Proposition \ref{prop:pushforwardceresa}  that shadows are always torsion, since the Jacobian of this genus 3 curve decomposes as a product of elliptic curve, and in each factor, the shadow is torsion.

\begin{proposition}
\label{prop:pushforwardceresa}
    Let $\pi : C \to E$, where $C$ is a smooth projective and geometrically integral curve, and $E$ is an elliptic curve. Let $\phi \in \End(J_C)$, where $J_C$ is the Jacobian of $C$. Then $\pi_* \Sh(\phi)$ is torsion (in fact $\pi_* \Sh(\phi) \in E[3]$).
\end{proposition}

\begin{proof}
    We briefly recall the setup for the shadow. Denote by $\Delta_C$ the diagonal in $C^2$, by $\Delta_C^3$ the small diagonal in $C^3$, by $\pr_i$ the projection onto the $i$-th factor, and by $\pr_{i,j}$ the projection onto the product of $i,j$ factors. Let $e = \frac{K_C}{2g-2}$, where $g$ is the genus of $C$, and $K_C$ its canonical divisor. Define
    $\Gamma^3(C,e)$ to be 
    $$
    \Delta_C^3 - \pr_{1,2}^*(\Delta_C) \cdot \pr_3^* e
    - \pr_{1,3}^*(\Delta_C) \cdot \pr_2^* e - \pr_{2,3}^*(\Delta_C) \cdot \pr_1^* e + \pr_{1,2}^*(e \times e) + \pr_{1,3}^*(e \times e) + \pr_{2,3}^*(e \times e).
    $$
    Writing $X_{\phi}$ for the correspondence induced by $\phi$ on $C^2$, we then have $\Sh(\phi) = (2g-2) S(X_{\phi})$, where $S_C$ is given by
    $$
    S_{C,e} : \CH_1(C^2) \xrightarrow{\pr_{1,2}^*} \CH_1(C^3)
    \xrightarrow{\cdot \Gamma^3(C,e)} \CH_1(C^3)
    \xrightarrow{\pr_{3,*}} \CH_1(C).
    $$
    The covering $\pi$ induces a commutative diagram
    $$
    \xymatrix{
    \CH_1(C^2) \ar[r]^{\pr_{1,2}^*} \ar[d]^{\pi_*} & \CH_1(C^3) \ar[r]^{\cdot \Gamma^3(C,e)} \ar[d]^{\pi_*} 
    & \CH_1(C^3) \ar[r]^{\pr_{3,*}} \ar[d]^{\pi_*}  & \CH_1(C) \ar[d]^{\pi_*}  \\
    \CH_1(E^2) \ar[r]^{\pr_{1,2}^*} & \CH_1(E^3) \ar[r]^{\cdot \Gamma^3(E,\pi_* e)} & \CH_1(E^3) \ar[r]^{\pr_{3,*}}
    & \CH_1(E).
    }
    $$
    However, note that since $\pi_*e$ is a degree 1 divisor on the elliptic curve $E$, by \cite[Corollary 4.7]{GS95}, we have $6\Gamma^3(E, \pi_* e) = 0$, hence $6S_{E,\pi_{*} e} = 0$, and in particular 
    \begin{equation*}
    3 \pi_* \Sh(\phi) = 3 (2g-2) \pi_* S_{C,e}(X_{\phi})
    = (g-1) \cdot 6S_{E, \pi_* e} (\pi_* X_{\phi}) = 0. \qedhere
    \end{equation*}
\end{proof}

\begin{corollary}
If $C$ is a Shimura curve such that the Jacobian of $C$ factors as a product of elliptic curves, then $\Sh(\phi)$ is torsion.
\end{corollary}

\begin{proof}
Let $J$ be the Jacobian of $C$ and suppose $J \sim \prod_i E_i$ is the isogeny decomposition. Then the composition of an Abel--Jacobi map $C \to J$ with a projection $J \to E_i$ yields a morphism $\pi_i :C \to E_i$.
Thus by Proposition \ref{prop:pushforwardceresa},  $\pi_{i*} \Sh(\phi)$ is torsion. So $\Sh(\phi)$ is torsion in $\prod_i E_i$ and thus in $J$.
\end{proof}

We end with some speculative computations about shadows of Hecke operators on Shimura curves.
Let $\phi = T_\ell$ for $\ell$ prime not dividing $DN$.
Then  $T_\ell = T_\ell^\vee$ and $\deg(T_\ell) = \ell+1$ so 
\begin{equation} \Sh(\phi) = (2g -2) F_{T_\ell} - \deg(F_{T_\ell}) K_X - 2 T_\ell(K_X) + 2 (\ell+1) K_X.
\end{equation}

\begin{proposition}
The set of fixed points $F_{T_\ell}$ is in bijection with the set of QM abelian surfaces $A$ such that there exists an imaginary quadratic order $R$ with \[R \hookrightarrow \Order_0(N) \hookrightarrow\End(A)\] and $R$ has an element of norm $\ell$.
\end{proposition}
\begin{proof}
Suppose $(A, \iota)$ is a fixed point of $T_{\ell}$. Then there exists a $(1, \ell)$-isogeny $\varphi : A \to A$ such that $\iota \circ \varphi = \varphi \circ \iota$. In particular, $A$ has an endomorphism of norm $\ell$.  
\end{proof}

\begin{proposition}
\label{prop:mult}
    Let $A$ be an abelian surface with QM by $\Order$ such that $R_n $, a quadratic order of discriminant $n$ and conductor $f$, embeds into $\calO$. Let $\ell \nmid DN$ be a prime. Then there are $h_{\ell}$ horizontal $(1,\ell)$-isogenies ($R_n \to R_n$), $a_{\ell}$ ascending $(1,\ell)$-isogenies ($R_n \to R_{n / \ell^2}$), and $d_{\ell}$ descending $(1,\ell)$ isogenies ($R_n \to R_{n \ell^2})$, where
    $$
    h_{\ell} = \begin{cases}
        0 & \ell \mid f \\
        1 + \left(\frac{n}{\ell}\right) & \ell \nmid f
    \end{cases},\quad
    a_{\ell} = \begin{cases}
        1 & \ell \mid f \\
        0 & \ell \nmid f
    \end{cases}, \quad 
    d_{\ell} = \begin{cases}
        \ell & \ell \mid f \\
        \ell - \left(\frac{n}{\ell}\right) & \ell \nmid f
    \end{cases}
    $$
\end{proposition}

\begin{proof}
For $D = 1$, this is \cite[\S 2.9, Lemma 6, Remark 8]{Sutherland} and \cite[Proposition 23]{Kohel}. 
For $D > 1$ and maximal $\Order$, this is \cite[Lemma 4.2, Lemma 4.3]{Saia}. However, since $\ell \nmid DN$, the proof works verbatim for an Eichler order.
\end{proof}

\begin{example}
In the case $\ell = 2$, and $(\ell, DN) = 1$  we have the following.

Write $D_n$ for the CM points corresponding to the order of discriminant $n$. Then 
\begin{equation}
\label{eq:FT2} F_{T_2} =   D_{-4} + D_{-8} + 2D_{-7}
\end{equation}

To compute this, consider the orders with discriminant $\Delta = t^2 - 4 \ell < 0$, where $t \in \Z$ and $\ell = 2$. This gives $\Delta = -4, -7, -8$.
 We look only at the horizontal isogenies, since these are maximal orders, to determine the points on the diagonal. The formula in Proposition \ref{prop:mult} gives the multiplicities.
\end{example}

In order to compute the shadow of a Hecke operator we need to be able to compute $T_\ell(D_{n})$. The following proposition addresses this.

\begin{proposition}
\label{prop:isogenies}
Let $R_n$ be the order of discriminant $n$ and $f$ the conductor of $R_n$. Let $D_n$ be  the divisor of CM points corresponding to the order of discriminant $n$, and let $\reldeg{n}{\ell} = \deg(D_n) / \deg(D_{n\ell^2})$. Then
\[ T_\ell(D_{n}) =  \begin{cases} 
2 D_n + (\ell - 1) \reldeg{n}{\ell} D_{n \ell^2} & \text{ if } \ell \text{ splits in } R_n\\
(\ell + 1) \reldeg{n}{\ell} D_{n \ell^2} & \text{ if } \ell \text{ is inert in } R_n \\
 \reldeg{n/\ell^2}{\ell}^{-1} D_{n/\ell^2} +  \ell \reldeg{n}{\ell} D_{n \ell^2} & \text{ if } \ell | f \\
  D_{n} +  \ell \reldeg{n}{\ell}   D_{n \ell^2} & \text{ if } \ell \text{ is ramified in } R_n, \ell \notdiv f 
\end{cases} \]
\end{proposition}
\begin{proof}
The operator $T_\ell$ sends a QM abelian surface $A$ in $D_n$ to the sum of the QM abelian surfaces that are $(1,\ell)$-isogenous. Based on the splitting behavior of $\ell$ in $R_n$, we can count the number of ascending,  descending, and  horizontal isogenies, from Proposition \ref{prop:mult}.
\end{proof}

\begin{example}
Consider the genus 5 curve $X = \Curve{15}{ 23}/\langle w_3, w_{23}\rangle$ with $T_2$. (It turns out that curve can be proven to be non-hyperelliptic using Proposition \ref{prop:fixedpointsAL}, but we will use this to illustrate the computations anyway.)

Using Proposition \ref{prop:localembed} to count optimal embeddings, we have that  $F_{T_2} = 2D_{-7}$ (the other CM orders in equation \eqref{eq:FT2} do not occur).
Consider the map  
$X \to X / \langle w_5 \rangle$. The latter is an elliptic curve. Using Proposition \ref{prop:fixedpointsAL} to compute the ramification, we find 
$K_X = D_{-15} + D_{-60} + D_{-115} + D_{-460} + D_{-1380}$ (with degrees $1, 1, 1, 3, 2$).
Using Proposition \ref{prop:isogenies}, we compute that
\begin{align*}
T_2(D_{-15}) &= 2D_{-15} + D_{-60}, &
T_2(D_{-60}) &= D_{-15} + D_{-240}, &
T_2(D_{-115}) = D_{-460},\\
T_2(D_{-460}) &= 3D_{-115} + D_{-1840}, &
T_2(D_{-1380}) &= D_{-1380} + D_{-5520}, &
\end{align*}
with divisors of degrees
\begin{align*}
\deg(D_{-15}) &= 1, & \deg(D_{-60}) &= 1, & \deg(D_{-115}) &= 1, &
\deg(D_{-240}) &= 2, \\
\deg(D_{-1380}) &= 2, &
\deg(D_{-460}) &= 3, & \deg(D_{-5520}) &= 4,  & \deg(D_{-1840}) &= 6.
\end{align*}
Therefore, we obtain
$$
T_2(K_X) = 3D_{-15} + D_{-60} + 3D_{-115} + D_{-240} + D_{-460} + D_{-1380} + D_{-1840} + D_{-5520},
$$
and it follows that we can express the shadow as the sum
\begin{align*}
\Sh(T_2) &= 16D_{-7} + 2K_X - 2T_2(K_X) \\
&= 16D_{-7} - 4D_{-15} -4D_{-115} - 2D_{-240} -2D_{-1840} - 2D_{-5520}.
\end{align*}
The basic problem is to decide if this divisor is non-torsion or not. 
\end{example}

\section{Models of low genus Atkin--Lehner quotients}
\label{sec:models}

In \cite{OanaFreddy} the authors study Atkin--Lehner quotients of Shimura curves of genus $0$, $1$, and $2$. 
They provide models for genus $1$ and $2$ curves in many cases, and often  determine if a rational point exists on these curves. 

In some cases the authors are unable to determine models, for example if they are unable to determine if a rational point exists on the curve in genus 0 and 1 \cite[Tables 1,6,7]{OanaFreddy} or if they are unable to determine if the curve is bielliptic in genus 2 \cite[Table 10]{OanaFreddy}. In the bielliptic case, they are also sometimes unable to decide between several models \cite[Remark 8.5]{OanaFreddy}.

\subsection{Equations of hyperelliptic Shimura curves}
The method of \cite{GY17} can be used to address some of these ambiguities.
Their idea is to use the exceptional isomorphism to $\Orth(2,1)$, and regard the Shimura curve as a Shimura variety for the orthogonal group, therefore amenable to construction of weakly holomorphic modular forms with prescribed divisors. Using Schofer's formula \cite{Schofer09}, one can then evaluate these functions at CM points. When a curve exhibits a double cover to $\P^1$, one can construct forms supported on the ramification divisors and by solving linear and quadratic equations arising from evaluation at CM points and then reconstruct equations for the curve. 

\subsection{Current limitations}
\label{sec:limits}
We have implemented the described algorithm, but ran into two difficulties. The first is a theoretical gap -- the evaluation of Schofer's formula at non-fundamental CM discriminants when the weight of the Borcherds form is not $0$ needs some modification. However, the correction term does not appear in the literature.

The second is in the implementation aspect. The proof relies on computing spaces of weakly holomorphic modular forms of level roughly $DN$ by spanning them with eta quotients. Determining the possible eta quotients is done via integer programming. For large values of $D$ and $N$, determining the possible eta quotients and narrowing this to a basis of eta quotients of sufficient precision around both $0$ and $\infty$ requires a massive investment of resources.

\subsection{Applications}

We give some examples to illustrate how our implementation can nonetheless be applied to missing cases in the aforementioned work. 
These examples are not exhaustive, but  give a flavor of what the code can do.  However, due to the limitations discussed in the previous section, the implementation would need serious improvements before being able to produce models for the majority of the missing cases.

\begin{proposition} \label{prop:rational genus 0}
For $(m,d) \in \{ (2,5), (5,7), (10,14) \}$, we have
\begin{equation}
    \Curve{10}{7}/\langle w_{m}, w_{d}\rangle \simeq \P^1_{\Q}.
\end{equation}
In particular, these genus $0$ curves all admit rational points.
\end{proposition}

\begin{proof}
The models for the three curves are (in order)  
   $ y^2 = -27/16x^2 -47/64x-5/64$, 
 $ y^2 = -27x^2 + 22x+ 5$,
and $y^2 = -27/64x^2 - 5/64x $
where $x$ is a Hauptmodul for the star curve $\StarCurve{10}{7}$.
All three conics have rational points.
\end{proof}

\begin{proposition}
The models in Table \ref{modelsgen2} and \ref{modelsbielliptic} are models for the listed genus 2 Shimura curves. 
The only geometrically bielliptic curve in Table \ref{modelsgen2} is $\Curve{14}{29}/\langle w_7, w_{29} \rangle$, which is bielliptic over $\Q$.
\end{proposition}

\begin{proof}
    The models were computed using the method of \cite{GY17}, and for each curve $X$, we computed its geometric automorphism group $\Aut_{\Qbar}(X)$ using the implementation in \Magma{} of \cite{AutGroup}. It turns out that $\Aut_{\Qbar}(X) \simeq \Z / 2 \Z$ for all $X$ in Table~\ref{modelsgen2} other than $X = \Curve{14}{29}/\langle w_7, w_{29} \rangle$, for which $\Aut_{\Q}(X) \simeq (\Z / 2 \Z)^2$. Quotienting by a non-hyperelliptic involution, we obtain the elliptic curve \href{https://www.lmfdb.org/EllipticCurve/Q/406/c/2}{406.c2} defined by 
    $$
    E_1 : y^2 + xy = x^3 + x^2 - 2124x - 60592,
    $$
    with $\pi_1 : X \to E_1$ given by 
    {\small
    \begin{equation*}
    (x,y) \mapsto \left( \frac{19208 x^3 - 278173 x^2 + 27167 x - 6130}{343 x^3 - 735 x^2 + 525 x - 125}, \frac{-9604x^3 + 19159 x^2 - 12873 x + 1421 y + 1644}{343 x^3 - 735 x^2 + 525 x - 125} \right)
    \end{equation*}}
    Quotienting by the other non-hyperelliptic involution, we get the elliptic curve \href{https://www.lmfdb.org/EllipticCurve/Q/14/a/2}{14.a2} defined by
    $$E_2 : y^2 + xy + y = x^3 - 171x - 874,$$
    with $\pi_2 : X \to E_2$ given by 
    {\small
    \[(x,y) \mapsto  \left( \frac{1792x^3 - 4385x^2 + 3604x - 996}{49x^2 - 84x + 36},  
    \frac{-768x^3 + 1912x^2 - 256xy - 1376x + 256y + 84}{49x^2 - 84x + 36}\right).\qedhere \]
    }
\end{proof}

\begin{table}
\centering
\small
\renewcommand{\arraystretch}{1.1}
\begin{tabular}{lll}
\hline
\textbf{Curve} & $h(x)$  & $f(x)$   \\
\hline

$\Curve{6}{29}/\langle w_3, w_{29} \rangle$
&
$-x^2-x$ & $-144x^5-117x^4+41x^3+21x^2-6x$

\\

$\Curve{14}{13}/\langle w_2, w_{13} \rangle$
&
$2x^3-3x^2+x$ &$\begin{aligned}[t]-44x^6&-626x^5-3296x^4-8298x^3\\
&-10950x^2-7306x-1950\end{aligned}$

\\

$\Curve{14}{13}/\langle w_7, w_{13} \rangle$
&
$2x^3-3x^2+x$ &$\begin{aligned}[t]-1716x^6&+1816x^5-117x^4-291x^3\\
&+11x^2+18x+2\end{aligned}$

\\

$\Curve{10}{17}/\langle w_{10}, w_{34} \rangle$
&
$x^2-1$ & $-4x^5+11x^4-33x^3+21x^2-15x-44$

\\

$\Curve{14}{29}/\langle w_7, w_{29} \rangle$
&
$-x^2+x-2$ &$1792x^5-6487x^4+9347x^3-6701x^2+2390x-340$

\\

$\Curve{14}{29}/\langle w_7, w_{58} \rangle$
&
$-x-1$&$-224x^5+365x^4-194x^3+47x^2-6x$

\\

$\Curve{6}{37}/\langle w_6, w_{74} \rangle$
&
$-x^2-x$ &$972x^5-2411x^4+2244x^3-929x^2+144x$

\\

$\Curve{6}{41}/\langle w_2, w_3 \rangle$
&
$-x^2-x$ & $\begin{aligned}[t]-2304x^6&-3862x^5-2270x^4-687x^3\\
&-230x^2-76x-10\end{aligned}$

\\

$\Curve{6}{41}/\langle w_3, w_{41} \rangle$
&
$-x^2-x$ &$-288x^5-303x^4-95x^3-27x^2-12x-2$

\\

$\Curve{6}{43}/\langle w_2, w_3 \rangle$
&
$-x^2-1$& $\begin{aligned}[t]-15552x^6&+24614x^5-16682x^4+5999x^3\\
&-1182x^2+119x-5\end{aligned}$

\\

$\Curve{10}{53}/\langle w_5, w_{106} \rangle$
&
$-x^2-x$ &$20x^5-25x^4+35x^3-16x^2+9x+1$

\\

$\Curve{6}{59}/\langle w_2, w_{59} \rangle$
&
$-x^2-x$ &$6x^5+16x^4-23x^3-26x^2+84x-40$

\\

$\Curve{6}{61}/\langle w_2, w_3 \rangle$
&
$x^3+x^2$ &$ -x^6-2x^5+12x^4+8x^3-41x^2+84x-232$
\\

$\Curve{106}{1}/\langle w_2 \rangle$
&
$-2x^3+3x^2-3x$ & $\begin{aligned}[t]-20x^6&+164x^5-540x^4+872x^3\\
&-700x^2+250x-32\end{aligned}$

\\

$\Curve{118}{1}/\langle w_2 \rangle$
&
$-2x^3-x^2$ &$\begin{aligned}[t]-688676x^6&+1519312x^5-1397890x^4+686641x^3\\ &-189920x^2+28048x-1728\end{aligned}$

\\

\hline
\end{tabular}
\caption{Polynomials $h(x)$ and $f(x)$ such that $y^2 +h(x)y =f(x)$ is a model of the curve for genus 2 curves that do not have bielliptic Atkin--Lehner involutions. Each curve is given as a degree $2$ cover of a star curve $\StarCurve{D}{N}$ with hauptmodul $x$.} 
\label{modelsgen2}
\end{table}

\begin{remark}
To make the equations look nicer, the equations  Table \ref{modelsgen2} corresponding to the same star curve  correspond to  different hauptmoduls.
\end{remark}

\begin{table}[ht]
\centering
\small
\renewcommand{\arraystretch}{1.1}
\begin{tabular}{lll}
\hline
\textbf{Curve} & $h(x)$  & $f(x)$   \\
\hline
$X_0(6,	17)/\langle w_3 \rangle$	&$ x^3 +1$& $ -5x^6+3x^5-44x^4-53x^3-44x^2+3x-5$\\

$X_0(6,	17)/\langle w_6 \rangle$	&0 & $-9x^6 + 27x^5 - 94x^4 + 143x^3 - 446x^2 + 379x - 321$\\
\hline
\end{tabular}
\caption{Polynomials $h(x)$ and $f(x)$ such that $y^2 +h(x)y =f(x)$ is a model of the curve for a genus 2 curve that is known to have a bielliptic Atkin--Lehner involution. The coordinate $x$ is a hauptmodul of  $\Curve{6}{17}/\langle w_3, w_{34} \rangle$ and $\Curve{6}{17}/\langle w_6, w_{34} \rangle$ in each case.}
\label{modelsbielliptic}
\end{table}

\newpage
\bibliographystyle{alpha}
\bibliography{arXiv_v1}

\appendix
\counterwithin{equation}{section}

\section{Trace formulas for the modular non-Atkin--Lehner involutions}
\label{appendix}

Recall that when either $4 \mid N$ or $9 \parallel N$, the curve $\Curve{D}{N}$ admits modular involutions which are non-Atkin--Lehner, as in Section \ref{sec:modularnonAL}. In particular, if $4 \parallel N$, then we have $\sigma = S_2 = \mattwo{2}{1}{0}{2}$ , if $8 \mid N$, then we have $\sigma = V_2 = S_2 w_{2^{v_2(N)}} S_2^{-1}$, and if $9 \parallel N$, we have $\sigma = V_3 = S_3 w_9 S_3^{-1}$. In either case, we have $\sigma \Gamma_0(N) \sigma^{-1} = \Gamma_0(N)$, so that $\Gamma_0(N) \sigma = \Gamma_0(N) \sigma \Gamma_0(N)$ is a double coset, giving rise to a double coset operator. In this appendix we derive some formulas, following \cite{PopaII}, that allow for faster trace formula computation for the operators $\sigma w_m$, where $w_m \in \ALfull{D}{N}$ is an Atkin--Lehner involution. 
These formulas allow one to compute the dimensions of the spaces described in Section \ref{sec:modularnonAL}.

Write $\Gamma = \Gamma_0(N)$, and let $\Sigma = \Gamma \sigma$ for some $\sigma$ such that $\sigma \Gamma \sigma^{-1} = \Gamma$.
For $M \in M_2(\Z)$, let $\calS_N(M, \Sigma) \colonequals \{ X \in \Gamma \backslash \Gamma_0(1) : XMX^{-1} \in \Sigma \}$, and $C_N(M, \Sigma) := \# \calS_N(M, \Sigma)$. 
By \cite[Theorem 1 and \S 2.3]{PopaII}, in order to compute $\Tr([\Sigma]; S_2(N))$, we need to be able to quickly evaluate $C_N(M, \Sigma)$, as was done for the Hecke operators by Oesterl{\'e} \cite[Lemma 4.1]{PopaII}. In this appendix, we describe how to do this computation.
This will require describing membership in $\Sigma$ by congruence conditions, which we split to two cases.

\begin{lemma} Assume $4 \parallel N$.
    Let $\Sigma = \Gamma S_2$, where $S_2 = \mattwo{2}{1}{0}{2}$.
    Then 
    $$
    M = \mattwo{a}{b}{c}{d} \in \Sigma \iff 
    \begin{array}{ccc}
        2N \mid c,  & 2 \mid a, & \det(M) = 4  \\
        4 \mid 2b - a,  & (a/2, N) = 1 & 4 \mid 2d -c
    \end{array}
    $$
\end{lemma}

\begin{proof}
    Assume first that $M \in \Sigma$, so that $MS_2^{-1}\in \Gamma$.
    Then $\det(M) = \det(MS_2^{-1})\det(S_2) = 4$, and 
    $$
    MS_2^{-1} = \frac{1}{4} \mattwo{2a}{2b-a}{2c}{2d-c} \in \Gamma_0(N).
    $$
    In particular, integrality of the entries yields $2 \mid a$,  $4 \mid 2b - a$ and $4 \mid 2d - c$, and the congruence condition yields $2N \mid c$, and $(a/2,N) = 1$. Conversely, if these all hold then all the entries are integral, $N \mid c/2$ and $\det(MS_2^{-1}) = 1$, showing that $MS_2^{-1} \in \Gamma_0(N)$.
\end{proof}

\begin{lemma} 
\label{lem:congruence conditions for trace}
    Let $p \in \{2,3\}$, and $p \nmid m \mid N$ such that $(m, N/m) = 1$ and $v = v_p(N) \ge 2$. Let $w_m$ be an Atkin--Lehner element of determinant $m$. 
    Let $\Sigma = \Gamma V_p w_m$, where $V_p = S_p w_{p^v} S_p^{-1}$ and $S_p = \mattwo{p}{1}{0}{p}$.
    Then $M = \mattwo{a}{b}{c}{d} \in \Sigma$ if and only if 
    \begin{align}
        \det(M) = mp^v, \quad & N \mid c, \quad m p^v \mid \tr(M) \\
        mp^v \mid a - \frac{c}{p}, \quad & \left(a - \frac{c}{p}, \frac{N}{mp^v} \right) = 1
        , \quad (b, m) = 1, \quad p \nmid b - \frac{d}{p}
    \end{align}
\end{lemma}

\begin{proof}
    First note that $S_p w_m S_p^{-1} w_m^{-1} \in \Gamma$, therefore
    $$
    \Gamma V_p w_m = \Gamma w_m V_p = \Gamma S_p w_m w_{p^v} S_p^{-1} = S_p (\Gamma w_{m p^v})S_p^{-1}.
    $$
    We have
    \begin{equation}
        \Gamma w_{mp^v} = \left \{
            M = \mattwo{a}{b}{c}{d} : \begin{array}{c} \det(M) = mp^v, \quad mp^v \mid \tr(M), \\ 
            \quad m p^v \mid a, \quad (a, \tfrac{N}{mp^v}) = 1, \quad (b, mp^v) = 1 \end{array}
        \right \},
    \end{equation}
    Since 
    $$
    S_p \mattwo{a}{b}{c}{d} S_p^{-1} = \mattwo{a + \frac{c}{p}}{b + \frac{d-a}{p} - \frac{c}{p^2}}{c}{d - \frac{c}{p}},
    $$
    we see that $M = \mattwo{a}{b}{c}{d} \in \Sigma$ if and only if $N \mid c$, $mp^v \mid \tr(M)$, $mp^v \mid a - \frac{c}{p}$, $\left(a - \frac{c}{p}, \frac{N}{mp^v} \right) = 1$ and $\left(b + \frac{a-d}{p} - \frac{c}{p^2}, mp^v \right) = 1$.
    Since $mp^v \mid a - \frac{c}{p}$, it follows that $mp^{v-1} \mid \frac{a}{p} - \frac{c}{p^2}$, and as $v \ge 2$, this last condition is equivalent to $\left(b - \frac{d}{p}, mp^v \right) = 1$. Moreover, $mp^{v} \mid N \mid c$, so $mp^{v-1} \mid \frac{c}{p}$, showing that $mp^{v-1} \mid a$. From $mp^v \mid \tr(M) = a + d$, it follows that also $mp^{v-1} \mid d$, hence this is equivalent to $(b,m) = 1$ and $p \nmid b - \frac{d}{p}$. 
    \end{proof}

Before we proceed, we also need a certain combinatorial counting lemma. 

\begin{lemma}
    Let $N$ be an integer. Let $A, B, C, D, \alpha \in \Z$, and set $u = \gcd(B,C,A-D,N)$.
    Assume that $\alpha^2 - (A + D) \alpha + (AD-BC) \equiv 0 \bmod Nu$.
    Let 
    \begin{equation} \label{eq: linear system mod N}
    S_{\alpha}(N,u) \colonequals \left \{
    (c:d) \in \P^1(\Z / N \Z) : 
    \mattwo{B}{\alpha - A}{D - \alpha}{-C}
    \vectwo{c}{d} \equiv 0 \bmod N
    \right \}.
    \end{equation}
    Then $$ \# S_{\alpha}(N,u) =  \frac{\psi(N)}{\psi(N/u)}, $$
    where $\psi(N) = \# \P^1( \Z / N \Z)$, as in \eqref{eqn:P1}.
\end{lemma}

\begin{proof}
    Assume that $N = p^a$, $u = p^b$ with $b \le a$. Let $v$ be the valuation with respect to $p$. 
    Note that
    $$
    p^{a+b} \mid \alpha^2 - (A+D) \alpha + (AD-BC) = -BC - (D-\alpha)(\alpha - A),
    $$
    hence $v(-BC - (D-\alpha)(\alpha - A)) \ge a + b \ge 2b$. 
    Since $v(B), v(C) \ge b$ we have $v(BC) \ge 2b$, so that 
    $$
    v(D - \alpha) + v(\alpha - A) = v((D - \alpha)(\alpha - A)) \ge 2b.
    $$
    Therefore, either $v(D - \alpha) \ge b$ or $v(\alpha - A) \ge b$.
    As $v(D - A) \ge b$ and $D - A = (D - \alpha) + (\alpha - A)$, we must have $v(D - \alpha) \ge b$ and $v(\alpha - A) \ge b$.
    It follows that \eqref{eq: linear system mod N} is equivalent to
    \begin{equation} \label{eq: linear system mod N/u}
    p^{-b}\mattwo{B}{\alpha - A}{D - \alpha}{-C}
    \vectwo{c}{d} \equiv 0 \bmod p^{a-b}
    \end{equation}
    The matrix $M = p^{-b}\mattwo{B}{\alpha - A}{D - \alpha}{-C}$ has $v(\det(M)) \ge a - b$.

    If $a = b$, then any $(c:d)$ is a solution, and we get $\psi(p^a) = p^a(1 + p^{-1})$ solutions.

    Assume $a > b$.
    Let $$M = X \mattwo{m_1}{0}{0}{m_2}Y$$ be the Smith normal form of $M$ with $X,Y \in \SL_2(\Z)$. By choice of $u$, we have $v(m_1) = 0$, so that $v(m_2) = v(\det(M)) \ge a - b$.
    $M$ induces a linear map $\overline{M}$ on $(\Z/ p^{a-b} \Z)^2$,
    and it follows that $\ker \overline{M} \simeq \Z / p^{a-b} \Z$. Thus, there is a unique $(c:d) \in \P^1(\Z / p^{a-b} \Z)$ solving \eqref{eq: linear system mod N/u}, lifting to $p^b$ points $(c:d) \in \P^1(\Z / N \Z)$ for \eqref{eq: linear system mod N}.

    Summing up, we have 
    $$
    \# S(p^a , p^b) = \begin{cases}
        p^a(1+p^{-1}) & a = b \\
        p^b & a > b
    \end{cases}
    = \frac{\psi(p^a)}{\psi(p^{a-b})}.
    $$
    Using CRT we obtain the result.
\end{proof}

The following proposition is then an analogue of \cite[Lemma 4.1]{PopaII} for $\sigma = V_p$.

\begin{proposition} Let $p \in \{2,3\}$, and let $v = v_p(N)$. Assume $v \ge 3$.
    For $M = \mattwo{A}{B}{C}{D}$ with $\tr(M) = t$, $\det(M) = p^v$, $u = (G(M),N)$, where $G(M) = \gcd(C, D-A, B)$, and $p^v \mid t$, let $Q_M(x,y) = Cx^2 + (D-A)xy -By^2$. Then
   \begin{equation}
    C_N(M,  \Gamma V_p) = \frac{\psi(N)}{\psi(N/u)} \# 
    \left \{ x \in \left( \Z / \tfrac{N}{p^{v-1}} \Z \right)^{\times} : 
    p^{v-2} x^2 - \tfrac{t}{p} x + 1 \equiv 0 \bmod \tfrac{N}{p^v} \cdot u\right\}
   \end{equation}
\end{proposition}

\begin{proof}
    Let $X = \mattwo{a}{b}{c}{d} \in \Gamma(1)$, then
    $$
     X M X^{-1} = \mattwo{\alpha}{-Q_M(-b,a)}{Q_M(-d,c)}{t-\alpha}, 
    $$
    where $\alpha$ satisfies
    $$
    \alpha^2 - t \alpha + p^v = Q_M(-d, c) Q_M(-b, a).
    $$
    By Lemma~\ref{lem:congruence conditions for trace} above, $XMX^{-1} \in \Gamma V_p$ if and only if 
    $$
    N \mid Q_M(-d,c), \quad p^v \mid \alpha - Q_M(-d,c)/p, \quad  p \nmid Q_M(-b,a), \quad (\alpha - Q_M(-d,c)/p, N/p^v) = 1.
    $$
    Assume $X \in \calS_N(M, \Gamma V_p)$. Then $N \mid Q_M(-d,c)$ and $u \mid Q_M(-b,a)$, hence 
    $$
    \alpha^2 - t \alpha + p^v \equiv 0 \bmod Nu.
    $$
    Since $p^{v-1} \mid Q_M(-d,c)/p$, we must have $p^{v-1} \mid \alpha$, and write $\alpha = p^{v-1} \alpha_0$, 
    giving us a map $\tau : C_N(M, \Sigma) \to S_N(u,t)$ sending $X \mapsto \alpha_0$, where
    $$
    S_N(u,t) = \{ \alpha_0 \in (\Z / (N/p^{v-1}) \Z) : p^{v-2} \alpha_0^2 - \tfrac{t}{p} \alpha_0 + 1 \equiv 0 \bmod \tfrac{N}{p^v} u\}.
    $$
    Moreover, $p^{v-1} \mid \alpha$ implies that $p^{2v-2} \mid \alpha^2 - t\alpha $. 
    If $v \ge 3$, then $2v-2 > v$, so that 
    $$ 
    Q_M(-d, c) Q_M(-b, a) = \alpha^2 - t \alpha + p^v \equiv p^v \bmod p^{v+1}, 
    $$
    showing that $p^{v+1} \nmid Q_M(-d,c)$. Since $p^v \mid \alpha - Q_M(-d,c)/p$, it follows that $p^v \nmid \alpha$, or equivalently $p \nmid \alpha_0$. Moreover, even if $v = 2$ and $p \mid \alpha_0$, we still get $p^4 \mid \alpha^2 - t \alpha$, hence $p \nmid \alpha_0$.
    Since $N \mid Q_M(-d,c)$, $N/p^v \mid Q_M(-d,c)/p$, so that $(\alpha, N/p^v) = 1$, showing that $\alpha_0 \in \left( \Z / \tfrac{N}{p^{v-1}} \Z \right)^{\times}$.
    
    Conversely, given $\alpha_0$, we recover $\alpha$, and every $(c:d) \in S_{\alpha}(N,u)$ gives rise to such $X$. The previous lemma now gives the result.
\end{proof}

The next proposition is the analogue for $\sigma = w_m V_p$ when $w_m$ and $V_p$ commute.

\begin{proposition} Let $p \in \{2,3\}$, $p \nmid m \mid N$ such that $(m, N/m) = 1$ and let $v = v_p(N)$. Assume $v \ge 2$. Moreover, if $p = 3$, assume that $m \equiv 1 \bmod 3$.
    For $M = \mattwo{A}{B}{C}{D}$ with $\tr(M) = t$, $\det(M) = mp^v$, $u = (G(M),N)$, where $G(M) = \gcd(C, D-A, B)$, and $mp^v \mid t$, let $Q_M(x,y) = Cx^2 + (D-A)xy -By^2$. Then
   \begin{equation}
    C_N(M,  \Gamma w_mV_p) = \frac{\psi(N)}{(p-1)\psi(N/u)} \# S_N(u,t)
   \end{equation}
   with
   $$
   S_N(u,t) = \left \{ x \in \left( \Z / \tfrac{N}{mp^{v-1}} \Z \right)^{\times} : 
    mp^{v-2} x^2 - \tfrac{t}{p} x + 1 \equiv 0 \bmod \tfrac{N}{mp^v} \cdot u\right\}.
   $$
\end{proposition}

\begin{proof}
    Let $X = \mattwo{a}{b}{c}{d} \in \Gamma(1)$, then
    $$
     X M X^{-1} = \mattwo{\alpha}{-Q_M(-b,a)}{Q_M(-d,c)}{t-\alpha}, 
    $$
    where $\alpha$ satisfies
    \begin{equation} \label{eq: min poly alpha}
    \alpha^2 - t \alpha + mp^v = Q_M(-d, c) Q_M(-b, a).
    \end{equation}
    By Lemma~\ref{lem:congruence conditions for trace} above, $XMX^{-1} \in \Gamma w_m V_p$ if and only if 
    \begin{align}
    N \mid Q_M(-d,c), \quad  mp^v \mid \alpha - \frac{Q_M(-d,c)}{p}, \quad & p \nmid Q_M(-b,a) + \frac{t - \alpha}{p}, \\
    \left (\alpha - \frac{Q_M(-d,c)}{p}, \frac{N}{mp^v} \right) = 1, \quad & (Q_M(-b,a), m) = 1.
    \end{align}
    Assume $X \in \calS_N(M, \Gamma W_{m} V_p)$. Then $N \mid Q_M(-d,c)$ and $u \mid Q_M(-b,a)$, hence 
    $$
    \alpha^2 - t \alpha + mp^v \equiv 0 \bmod Nu.
    $$
    Since $mp^{v-1} \mid Q_M(-d,c)/p$, we must have $mp^{v-1} \mid \alpha$, and write $\alpha = mp^{v-1} \alpha_0$, 
    giving us a map $\tau : C_N(M, \Sigma) \to S_N(u,t)$ sending $X \mapsto \alpha_0$, where
    $$
    S_N(u,t) = \{ \alpha_0 \in (\Z / (\tfrac{N}{mp^{v-1}}) \Z) : mp^{v-2} \alpha_0^2 - \tfrac{t}{p} \alpha_0 + 1 \equiv 0 \bmod \tfrac{N}{mp^v} u\}.
    $$
    Moreover, $mp^{v-1} \mid \alpha$ implies that $mp^{2v-2} \mid \alpha^2 - t\alpha $. 
    If $v \ge 3$, then $2v-2 > v$, so that 
    $$ 
    Q_M(-d, c) Q_M(-b, a) = \alpha^2 - t \alpha + mp^v \equiv mp^v \bmod mp^{v+1}, 
    $$
    showing that $p^{v+1} \nmid Q_M(-d,c)$. Since $p^v \mid \alpha - Q_M(-d,c)/p$, it follows that $p^v \nmid \alpha$, or equivalently $p \nmid \alpha_0$. Moreover, even if $v = 2$ and $p \mid \alpha_0$, we still get $p^4 \mid \alpha^2 - t \alpha$, hence $p \nmid \alpha_0$.
    Furthermore, as $(\alpha_0 - Q_M(-d,c) / p, N/m) = 1$, and $N/m \mid Q_M(-d,c)/p$, it follows that $(\alpha_0, N/m) = 1$, hence $(\alpha_0, \tfrac{N}{mp^{v-1}}) = 1$. 
    
    Conversely, given $\alpha_0$, we recover $\alpha$. 
    Note that if $(u, mp^v) \ne 1$, then $S_N(u,t) = \emptyset$ in all cases, so we may also assume that $(u, mp^v) = 1$.
    
    From \eqref{eq: min poly alpha}, by taking the valuation at $p$ on both sides, we get $v = v_p(Q_M(-d,c)) + v_p(Q_M(-b,a))$, and since $v \le v_p(Q_M(-d,c))$, it follows that $p \nmid Q_M(-b,a)$ and $p^{v+1} \nmid Q_M(-d,c)$.  

    If $p = 2$ (and then $v \ge 3$), then this shows $mp^v \mid \alpha - \tfrac{Q_M(-d,c)}{p}$. 
    Since $v \ge 3$, $p^2 \mid t - \alpha$, and $p \nmid Q_M(-b,a)$, we also get $p \nmid Q_M(-b,a) - \tfrac{t - \alpha}{p}$.
    Therefore, in this case, every $(c:d) \in S_{\alpha}(N,u)$ gives rise to $X \in \calS_N(\Gamma, \Sigma)$. The previous lemma now gives the result.

    If $p = 3$ and $v = 2$, then every $\alpha_0 \in S_N(u,t)$ has a conjugate $\overline{\alpha}_0 \in S_N(u,t)$ satisfying $3 \mid \tfrac{t}{p} = \alpha_0 + \overline{\alpha}_0$. Therefore, exactly one of $\alpha_0, \overline{\alpha}_0$ satisfies $p \mid \alpha_0 - \tfrac{Q_M(-d,c)}{mp^v}$, hence $mp^v \mid \alpha - \frac{Q_M(-d,c)}{p}$. 

    Finally, since $m \equiv 1 \bmod 3$, dividing \eqref{eq: min poly alpha} by $3m$ and considering it modulo $3$, using that $3 \mid \alpha_0 - Q_M(-d,c)/3m$, we get that $\alpha_0^2 + 1 \equiv \alpha_0 Q_M(-b,a) \bmod 3$, hence $Q_M(-b,a) \equiv -\alpha_0 \bmod 3$, so that $3 \nmid Q_M(-b,a) + \alpha / 3$.
    Thus, in this case, we get a single solution for every conjugate pair.
\end{proof}

\end{document}